\documentclass[a4paper]{article}

\usepackage[margin=1in]{geometry} %

\usepackage{amsmath}
\usepackage{amsthm}
\usepackage{amssymb}
\usepackage{graphicx}
\usepackage{xcolor}
\usepackage{mathrsfs} %
\usepackage[utf8]{inputenc}
\usepackage{adjustbox}

\usepackage{tikz}
\usetikzlibrary{shapes,calc,math,backgrounds,matrix}

\usepackage{hyperref}

\usepackage[sort&compress,numbers,square]{natbib}
\definecolor{DAcolor}{rgb}{0.1,0.1,1}

\definecolor{DAHcolor}{rgb}{0.9,0.1,0.1}

\newcommand{\Tk}{\mathsf{TS}_k}
\newcommand{\Tw}{\mathsf{TS}_2}

\newcommand{\Tkp}{\mathsf{TS}_{k+1}}

\theoremstyle{plain}
\newtheorem{theorem}{Theorem}
\newtheorem{corollary}[theorem]{Corollary}
\newtheorem{lemma}[theorem]{Lemma}

\newtheorem{conjecture}[theorem]{Conjecture}

\newtheorem{proposition}[theorem]{Proposition}

\theoremstyle{definition}
\newtheorem{definition}[theorem]{Definition}

\newtheorem{problem}[theorem]{Problem}

\title{\textbf{A Note On Acyclic Token Sliding Reconfiguration Graphs of Independent Sets}}
\author{David~Avis$^1$ \and Duc~A.~Hoang$^2$}
\date{
	$^1$ Graduate School of Informatics, Kyoto University, Japan\\
	School of Computer Science, McGill University, Canada\\
	\texttt{avis@cs.mcgill.ca}\\
	$^2$ Graduate School of Informatics, Kyoto University, Japan\\ 
	\texttt{hoang.duc.8r@kyoto-u.ac.jp}
}

\begin{document}
\maketitle

\begin{abstract}
We continue the study of token sliding reconfiguration graphs of independent sets initiated
by the authors in an earlier paper (arXiv:2203.16861). Two of the topics in that paper were to
study which graphs $G$ are token sliding graphs and which properties of a graph are
inherited by a token sliding graph. 
In this paper we continue this study specializing on the case of when $G$
and/or its token sliding graph $\mathsf{TS}_k(G)$ is a tree or forest,
where $k$ is the size of the independent sets considered. We consider two problems.
The first is to find necessary and sufficient conditions on $G$ for $\mathsf{TS}_k(G)$ to be a forest.
The second is to find necessary and sufficient conditions for a tree or forest to
be a token sliding graph. For the first problem we give a forbidden
subgraph characterization for the cases of $k=2,3$. For the second problem
we show that for every $k$-ary tree $T$ there is a graph $G$ for which $\mathsf{TS}_{k+1}(G)$
is isomorphic to $T$. A number of other results are given along with a join operation
that aids in the construction of $\mathsf{TS}_k(G)$-graphs.
\end{abstract}

\section{Introduction}
\label{sec:intro}

In a \emph{reconfiguration variant} of a computational problem (e.g., \textsc{Satisfiability}, \textsc{Independent Set}, \textsc{Vertex-Coloring}, etc.), a \emph{transformation rule} that describes an \emph{adjacency relation} between \emph{feasible solutions} (e.g., satisfying truth assignments, independent sets, proper vertex-colorings, etc.) of the problem is given.
One of the main goals is to decide whether there is a sequence of adjacent feasible solutions that ``reconfigures'' one given solution into another.
Another way of looking at these reconfiguration problems is via the so-called \emph{reconfiguration graph}---a graph whose nodes are feasible solutions and two nodes are adjacent if one can be obtained from the other by applying the given rule exactly once.
The mentioned question now becomes deciding whether there is a path between two given nodes in the reconfiguration graph.
Recently, \emph{reconfiguration problems} have been intensively studied from different perspectives~\cite{Heuvel13,MynhardtN19,Nishimura18,BousquetMNS22}.

One of the most well-studied reconfiguration variants of \textsc{Independent Set} is the so-called \textsc{Token Sliding} problem, which was first introduced by Hearn and Demaine~\cite{HearnD05} in 2005.
We refer readers to~\cite{Heuvel13,Nishimura18,BousquetMNS22} and the references therein for more details.
Surprisingly, though \textsc{Token Sliding} has been well-investigated, the realizability and structural properties of its corresponding reconfiguration graph---the one which we will refer to as the \emph{$\mathsf{TS}_k$-graph} (which stands for \emph{Token Sliding (Reconfiguration) graph})---have not been studied until recently~\cite{AvisH22}.
On the other hand, when considering either general vertex subsets, dominating sets, or proper vertex-colorings of a graph as the ``input feasible solutions'', their corresponding reconfiguration graphs have been very well-characterized~\cite{MonroyFHHUW12,MynhardtN19}.

For any graph-theoretic terminology and notation not defined here, we refer readers to~\cite{Diestel2017}.
Given a graph $G = (V, E)$ and an integer $k \geq 2$.
For two sets $X, Y$, we sometimes use $X + Y$ and $X - Y$ to indicate $X \cup Y$ and $X \setminus Y$.
We abbreviate $X \cup \{u\}$ (resp., $X \setminus \{u\}$)
by $X + u$ (resp., $X - u$).
We use $N_G(u)$, or simply just $N(u)$ when the graph $G$ is clear from the context, to denote the \emph{(open) neighbors} of $u$, i.e., set of all vertices in $G$ that are adjacent to $u$.
The \emph{closed neighbors} of $u$, denoted by $N_G[u]$ or simply $N[u]$, is the set $N_G(u) + u$.
The \emph{degree} of $u$, denoted by $\deg_G(u)$, is nothing but the size of $N_G(u)$.
An \emph{independent set} (or \emph{stable set}) of $G$ is a vertex subset $I$ such that for every $u, v \in I$ we have $uv \notin E(G)$.
The \emph{$\mathsf{TS}_k$-graph} of $G$, denoted by $\mathsf{TS}_k(G)$, takes all size-$k$ independent sets of $G$ as its nodes and two nodes $I, J$ are \emph{adjacent (under Token Sliding ($\mathsf{TS}$))} if there exist two vertices $u, v \in V(G)$ such that $I - J = \{u\}$, $J - I = \{v\}$, and $uv \in E(G)$.
Two graphs $G$ and $H$ are \emph{isomorphic}, denoted by $G \simeq H$, if there exists a bijective mapping $f: V(G) \to V(H)$ such that $uv \in E(G)$ if and only if $f(u)f(v) \in E(H)$.
A graph $G$ is called a \emph{$\mathsf{TS}_k$-graph} if there exists a graph $H$ such that $G \simeq \mathsf{TS}_k(H)$.
A \emph{forest} is a graph having no cycles (i.e., it is \emph{acyclic}) and a connected forest is a \emph{tree}.
A \emph{$\mathsf{TS}_k$-tree/forest} is a $\mathsf{TS}_k$-graph which is also a tree/forest.
\figurename~\ref{D132} illustrates a $\mathsf{TS}_2$-tree on six vertices (right).
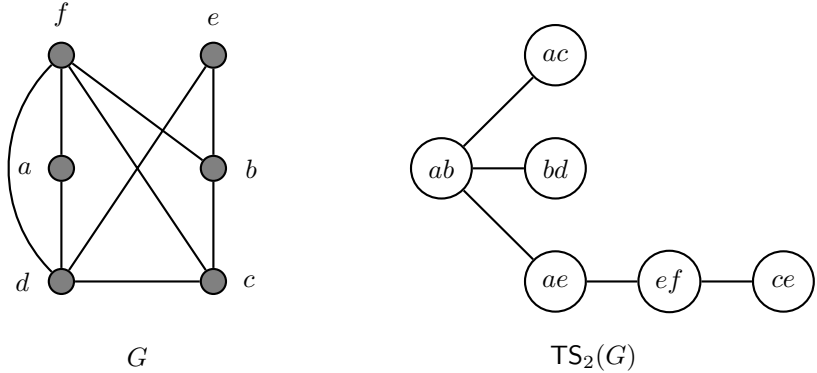
\begin{figure}[!ht]
	\centering
	\begin{tikzpicture}[every node/.style={circle, draw, thick, minimum size=0.8cm}]
		\begin{scope}[shift={(5,0)}]
			\node (v1) at (0,0) {$ab$};
			\node (v2) at (1.5,1.5) {$ac$};
			\node (v3) at (1.5,0) {$bd$};
			\node (v4) at (1.5,-1.5) {$ae$};
			\node (v5) at (3,-1.5) {$ef$};
			\node (v6) at (4.5,-1.5) {$ce$};
			
			\draw[thick] (v1) -- (v2) (v1) -- (v3) (v1) -- (v4) -- (v5) -- (v6);
			
			\node[draw=none] (T) at (2,-2.5) {$ \mathsf{TS}_2(G)$};
		\end{scope}
		\begin{scope}[every node/.style={circle, draw, thick, fill=gray, minimum size=0.3cm}]
			\node[label=left:$a$] (a) at (0,0) {};
			\node[label=right:$b$] (b) at (2,0) {};
			\node[label=right:$c$] (c) at (2,-1.5) {};
			\node[label=left:$d$] (d) at (0,-1.5) {};
			\node[label=above:$e$] (e) at (2,1.5) {};
			\node[label=above:$f$] (f) at (0,1.5) {};
			
			\draw[thick] (b) -- (c) (a) -- (d) (d) -- (c) (e) -- (b) (e) -- (d) (a) -- (f) (b) -- (f) (c) -- (f) (d) edge[bend left=45] (f);
			
			\node[draw=none, fill=none] (G) at (1,-2.5) {$G$};
		\end{scope}
	\end{tikzpicture}
	\caption{A graph $G$ with $\mathsf{TS}_2(G) = D_{1,3,2}$. Each node $ab$ represents a size-$2$ stable set of $G$.}
	\label{D132}
\end{figure}
In~\cite{AvisH22}, the authors studied various properties of the family of $\mathsf{TS}_k$-graphs. For a graph $G$, two of the questions studied were:
\begin{itemize}
	\item[(Q1)] What are necessary and sufficient conditions for $G$ so that $\mathsf{TS}_k(G)$ is a forest?
	\item[(Q2)] What are necessary and sufficient conditions for $G$ to be a $\mathsf{TS}_k$-graph?
\end{itemize}
In this paper, we study these two questions for the case when $G$ is a tree or a forest.

The \emph{union} $G \cup H$ of two (labelled) graphs $G$ and $H$ is the graph with $V(G \cup H) = V(G) \cup V(H)$ and $E(G \cup H) = E(G) \cup E(H)$.
When vertices and edges of $G$ and $H$ are considered distinct regardless of their labels, we say that $G \cup H$ is the \emph{disjoint union} of $G$ and $H$, and write $G + H$ instead of $G \cup H$ to distinguish from their union.
We respectively denote by $K_n$, $P_n$, and $C_n$ the \emph{complete graph}, \emph{path}, and \emph{cycle} on $n$ vertices. 
$K_{m,n}$ ($m \leq n$) is the \emph{complete bipartite graph} whose two partite sets are of sizes $m$ and $n$ respectively.
$K_{1,n}$ is also called a \emph{star}---a tree obtained by attaching $n$ leaves to a central vertex.
A family of graphs that we will use in the sequel generalizes 
stars and paths. 
For fix integers $n,r,s \ge 1$, let $D_{r,n,s}$ be the tree obtained from $P_n$
by appending $r$ leaves at one end and $s$ leaves at the other.
Note that $D_{1,1,s}$ is the star $K_{1,s+1}$ and
$D_{1,n,1}$ is the path $P_{n+2}$. 
\figurename~\ref{D132} illustrates $D_{1,3,2}$ (right).
An \emph{$n$-ary tree} is a rooted tree in which each node has at most $n$ children.
Any tree with maximum degree at most $n+1$ can be rooted at a vertex
with degree at most $n$ (e.g., a leaf) to produce a $n$-ary tree.
In particular, a $2$-ary tree is nothing but the well-known binary tree.

In the next section, we begin by partially answering (Q1) when $G$ is a tree/forest and $k \in \{2,3\}$ and conclude the section by conjecturing for $k \geq 4$.
Then, before addressing (Q2) for some trees/forests, in particular $k$-ary trees and $D_{r,n,s}$, we define an important graph operation which, under certain conditions, can be used for combining two $\mathsf{TS}_k$-graphs by taking their union to obtain a new one.
The final section of the paper gives some concluding remarks.

\section{Results on (Q1)}

\begin{figure}[htpb]
	\centering
	\begin{tikzpicture}
		\matrix (M) [matrix of nodes, 
		nodes in empty cells, 
		ampersand replacement=\&, 
		row sep=-\pgflinewidth, 
		column sep=-\pgflinewidth, 
		every row/.style={nodes={text width=10em, align=center}}, 
		column 1/.style={nodes={draw, rectangle, thick, minimum width=4cm}},
		column 2/.style={nodes={draw, rectangle, thick, minimum width=5.2cm}},
		column 3/.style={nodes={draw, rectangle, thick, minimum width=4.7cm}},
		row 1/.style={nodes={text height=0.3cm}},
		row 6/.style={nodes={text height=2.1cm}},
		row 8/.style={nodes={text height=2.1cm}},
		row 9/.style={nodes={text height=2.8cm}},
		row 10/.style={nodes={text height=2.1cm}},
		row 11/.style={nodes={text height=2.1cm}}
		] {
			$G$\vphantom{$\mathsf{TS}_2(G)$} 
			\& $\mathsf{TS}_2(G)$\\
			\tikz[every node/.style={draw, circle, thick, fill=white, minimum width=1em}]{
				\foreach \i/\x/\y in {0/0/0,1/1/0,2/0/-1,3/1/-1} {
					\node (\i) at (\x, \y) {};
				}
				\draw[thick] (0) -- (2) (1) -- (3);
			} 
			\& \tikz[every node/.style={draw, circle, thick, fill=gray, minimum width=1em}]{
				\foreach \i/\x/\y in {01/0/0,12/0/-1,23/1/-1,03/1/0} {
					\node (\i) at (\x, \y) {};
				}
				\draw[thick] (01) -- (12) -- (23) -- (03) -- (01);
			}\\
			\tikz[every node/.style={draw, circle, thick, fill=white, minimum width=1em}]{
				\foreach \i/\x/\y in {0/0/0,1/1/0,2/0/-1,3/1/-1} {
					\node (\i) at (\x, \y) {};
				}
				\draw[thick] (0) -- (2) -- (3) -- (0);
			} 
			\& \tikz[every node/.style={draw, circle, thick, fill=gray, minimum width=1em}]{
				\foreach \i/\x/\y in {01/0/0,12/0/-1,13/1/-1} {
					\node (\i) at (\x, \y) {};
				}
				\draw[thick] (01) -- (12) -- (13) -- (01);
			}\\
			\tikz[every node/.style={draw, circle, thick, fill=white, minimum width=1em}]{
				\foreach \i/\x/\y in {0/0/0,1/1/-1,2/2/-0.5,3/1/0,4/0/-1} {
					\node (\i) at (\x, \y) {};
				}
				\draw[thick] (0) -- (3) -- (1) -- (4) -- (0);
			}
			\& \tikz[every node/.style={draw, circle, thick, fill=gray, minimum width=1em}]{
				\foreach \i/\x/\y in {02/0/0,24/1/0,12/1/-1,23/0/-1,01/2/0,34/2/-1} {
					\node (\i) at (\x, \y) {};
				}
				\draw[thick] (02) -- (24) -- (12) -- (23) -- (02);
			}\\
			\tikz[every node/.style={draw, circle, thick, fill=white, minimum width=1em}]{
				\foreach \i/\x/\y in {0/0/0,1/0/-1,2/3/0,3/3/-1,4/1/-0.5,5/2/-0.5} {
					\node (\i) at (\x, \y) {};
				}
				\draw[thick] (4) -- (0) (4) -- (1) (4) -- (5) (5) -- (2) (5) -- (3);
			}
			\& \tikz[every node/.style={draw, circle, thick, fill=gray, minimum width=1em}]{
				\foreach \i/\x/\y in {02/0/0,24/1/0,12/2/0,15/3/0,13/3/-1,34/2/-1,03/1/-1,05/0/-1,01/4/0,23/4/-1} {
					\node (\i) at (\x, \y) {};
				}
				\draw[thick] (02) -- (24) -- (12) -- (15) -- (13) -- (34) -- (03) -- (05) -- (02);
			}\\
			\tikz[baseline=-1.5cm,every node/.style={draw, circle, thick, fill=white, minimum width=1em}]{
				\foreach \i/\x/\y in {0/0/-0.5,5/1/0,4/1/-1,3/-1/-0.5,2/2/0,1/2/-1} {
					\node (\i) at (\x, \y) {};
				}
				\draw[thick] (0) -- (5) -- (4) -- (0) (0) -- (3) (5) -- (2) (4) -- (1);
			}
			\& \tikz[baseline=-1.5cm,every node/.style={draw, circle, thick, fill=gray, minimum width=1em}]{
				\foreach \i/\x/\y in {01/0/0,13/1/0,34/2/0,35/3/0,23/4/-0.5,02/3/-1,24/2/-1,12/1/-1,15/0/-1} {
					\node (\i) at (\x, \y) {};
				}
				\draw[thick] (01) -- (13) -- (34) -- (35) -- (23) -- (02) -- (24) -- (12) -- (15) -- (01);
			}\\
			\tikz[every node/.style={draw, circle, thick, fill=white, minimum width=1em}]{
				\foreach \i/\x/\y in {0/0/0,4/1/0,2/2/0,1/2/-1,5/1/-1,3/0/-1} {
					\node (\i) at (\x, \y) {};
				}
				\draw[thick] (1) -- (5) -- (3) -- (0) -- (4) -- (2) (4) -- (5) (2) -- (5) (1) -- (4);
			}
			\& \tikz[every node/.style={draw, circle, thick, fill=gray, minimum width=1em}]{
				\foreach \i/\x/\y in {01/0/0,05/1/0,02/2/0,23/2/-1,34/1/-1,13/0/-1,12/3/-0.5} {
					\node (\i) at (\x, \y) {};
				}
				\draw[thick] (01) -- (05) -- (02) -- (23) -- (34) -- (13) -- (01) ;
			}\\
			\tikz[every node/.style={draw, circle, thick, fill=white, minimum width=1em}]{
				\foreach \i/\x/\y in {0/0/0,3/1/0,1/2/0,5/2/-1,2/1/-1.7,4/0/-1} {
					\node (\i) at (\x, \y) {};
				}
				\draw[thick] (0) -- (3) -- (1) -- (5) -- (2) -- (4) -- (0) (3) -- (4) -- (5) -- (3);
			}
			\& \tikz[baseline=-1.5cm,every node/.style={draw, circle, thick, fill=gray, minimum width=1em}]{
				\foreach \i/\x/\y in {01/0/0,05/1/0,02/2/0,23/2/-1,12/1/-1,14/0/-1} {
					\node (\i) at (\x, \y) {};
				}
				\draw[thick] (01) -- (05) -- (02) -- (23) -- (12) -- (14) -- (01);
			}\\
			\tikz[baseline=-0.7cm,every node/.style={draw, circle, thick, fill=white, minimum width=1em}]{
				\foreach \i/\x/\y in {0/0/0,4/1/0,6/1/1,5/0/1,1/0.5/1.7,3/2/1,2/-1/1} {
					\node (\i) at (\x, \y) {};
				}
				\draw[thick] (0) -- (5) -- (1) -- (6) -- (4) -- (0) (2) -- (5) -- (6) -- (3);
			}
			\& \tikz[every node/.style={draw, circle, thick, fill=gray, minimum width=1em}]{
				\foreach \i/\x/\y in {01/0/0,14/1/0,45/2/0,24/3/0,26/4/-0.5,23/3/-1,35/2/-1,03/1/-1,06/0/-1,02/3/0.7,12/4/0.3,13/2/-1.7,34/1/-1.7} {
					\node (\i) at (\x, \y) {};
				}
				\draw[thick] (01) -- (14) -- (45) -- (24) -- (26) -- (23) -- (35) -- (03) -- (06) -- (01) (24) -- (02) (26) -- (12) (35) -- (13) (03) -- (34);
			}\\
			\tikz[every node/.style={draw, circle, thick, fill=white, minimum width=1em}]{
				\foreach \i/\x/\y in {5/0/0,1/1/0,4/2/0.5,0/1/1,3/0.5/1.7,6/0/1,2/-1/0.5} {
					\node (\i) at (\x, \y) {};
				}
				\draw[thick] (5) -- (1) -- (4) -- (0) -- (3) -- (6) -- (2) -- (5) (0) -- (5) (0) -- (6) (1) -- (6);
			}
			\& \tikz[baseline=-0.6cm,every node/.style={draw, circle, thick, fill=gray, minimum width=1em}]{
				\foreach \i/\x/\y in {13/0/0,34/1/0,46/2/0,24/3/0.5,02/2/1,23/1/1,35/0/1,01/-1/0,12/3/1.2,45/3/-0.2,56/-1/1} {
					\node (\i) at (\x, \y) {};
				}
				\draw[thick] (13) -- (34) -- (46) -- (24) -- (02) -- (23) -- (35) -- (13) (13) -- (01) (24) -- (12) (24) -- (45) (35) -- (56);
			}\\
			\tikz[every node/.style={draw, circle, thick, fill=white, minimum width=1em}]{
				\foreach \i/\x/\y in {4/0/0,0/1/0,3/2/0.5,5/1/1,2/0.5/1.7,6/0/1,1/-1/0.5} {
					\node (\i) at (\x, \y) {};
				}
				\draw[thick] (0) -- (3) -- (5) -- (2) -- (6) -- (1) -- (4) -- (0) (0) -- (5) (1) -- (5) (3) -- (5) (3) -- (6) (5) -- (6) (4) -- (6);
			}
			\& \tikz[baseline=-0.6cm,every node/.style={draw, circle, thick, fill=gray, minimum width=1em}]{
				\foreach \i/\x/\y in {02/0/0,06/1/0,01/2/0,13/3/0.5,34/2/1,45/1/1,24/0/1,23/-1/0,12/-1/1} {
					\node (\i) at (\x, \y) {};
				}
				\draw[thick] (02) -- (06) -- (01) -- (13) -- (34) -- (45) -- (24) -- (02) (02) -- (23) (24) -- (12);
			}\\
		};
	\end{tikzpicture}
	\caption{A list $\mathcal{G}$ of $n$-vertex graphs $G$ ($4 \leq n \leq 7$) excluding $\overline{C_n}$ ($n \geq 5$) such that if $\mathsf{TS}_2(G^\prime)$ has no cycle then $G^\prime$ does not contain any member $G$ of $\mathcal{G}$ as an induced subgraph.}%
	\label{fig:acyclicForbiddenSubgraph}
\end{figure}

In this section, we prove the necessary and sufficient conditions on a tree/forest $G$ such that $\mathsf{TS}_k(G)$ is acyclic for $k \in \{2,3\}$, partially answering (Q1).

We begin with some definitions and observations.
The \emph{complement} $\overline{G}$ of a graph $G$ is the graph with $V(\overline{G}) = V(G)$ and $E(\overline{G}) = \{uv: uv \notin E(G)\}$. 
The \emph{size-$m$ matching}, denoted by $mK_2$, is the graph obtained by taking the disjoint union of $m$ copies of $K_2$. Observe that $\mathsf{TS}_2(2K_2) \simeq C_4$.
We label vertices in a $D_{r,n,s}$ ($r, s \geq 1$) as follows: 
Vertices of $P_n$ are labelled $p_1, \dots, p_n$.
The $r$ leaves attached to $p_1$ are $u_1, \dots, u_r$ and the $s$ leaves attached to $p_n$ are $v_1, \dots, v_n$.
$D_{2,2,2}$ is shaped like an \textbf{H} and $\mathsf{TS}_2(D_{2,2,2})$ contains a cycle
$C_8$ whose vertex-set is $\{u_1v_1, u_1p_2, u_1v_2, p_1v_2, u_2v_2, u_2p_2, u_2v_1, p_1v_1\}$.
Indeed, respectively from Lemma~1 of \cite{AvisH22} and~\figurename~\ref{fig:acyclicForbiddenSubgraph}, if a $n$-vertex graph $G$ is either $\overline{C_n}$ ($n \geq 5$) or a graph in the list $\mathcal{G}$ described in~\figurename~\ref{fig:acyclicForbiddenSubgraph} (which includes $2K_2$ and $D_{2,2,2}$), the graph $\mathsf{TS}_2(G)$ contains a cycle.
Additionally, we have:

\begin{lemma}
\label{db}

\begin{itemize}
\item[(a)] 
For $k \ge 2$, $\mathsf{TS}_k(2K_2+nK_1)$ contains a cycle $C_4$ if $n \ge k-2$ 
otherwise it is acyclic.
\item[(b)] 
	For $k \in \{2,3\}$, $s \ge 1$, $\mathsf{TS}_k(D_{1,n,s})$ contains a cycle $C_{4}$ if $n \ge 2k-1$
otherwise it is acyclic.
\item[(c)] 
For $k \in \{2,3\}$ and $r,s \ge 2$, $\mathsf{TS}_k(D_{r,n,s})$ contains a cycle $C_8$ if $n \ge 2k-2$
otherwise it is acyclic.
\end{itemize}
\end{lemma}
\begin{proof}
\begin{itemize}
	\item[(a)] If $n < k - 2$, there is no size-$k$ independent set in $2K_2 + nK_1$, thus its $\mathsf{TS}_k$-graph is obviously acyclic.
	Otherwise, let $I \subseteq V(nK_1)$ be an arbitrary independent set of size $k-2$, and let $E(2K_2) = \{ab, cd\}$. 
	Then, $\{I+a+c, I+a+d, I+b+c, I+b+d\}$ induce a $C_4$ in $\mathsf{TS}_k(2K_2 + nK_1)$.
	
	\item[(b)] Observe that if $n \geq 2k - 1$, $D_{1,n,s}$ contains an induced $2K_2 + (k-2)K_1$, which can be obtained by taking $u_1p_1$ and $p_nv_1$ as edges of $2K_2$ and the remaining $k-2$ independent vertices from the path $D_{1,n,s} - \{u_1, p_1, p_2, p_{n-1}, p_n, v_1, \dots, v_s\}$ on $n - 4$ vertices.
	(Since $n \geq 2k - 1$, this path has an independent set of size at least $\lceil (n-4)/2 \rceil \geq \lceil (2k - 5)/2 \rceil = k - 2$.)
	Then, using a similar argument as in (a) we have $\mathsf{TS}_k(D_{1,n,s})$ contains a $C_4$.  
	
	On the other hand, if $n \leq 2k - 2$ for $k \in \{2,3\}$, since $D_{1,n-1,s}$ is always an induced subgraph of $D_{1,n,s}$ for $n \geq 2$, it follows that if $\mathsf{TS}_2(D_{1,n-1,s})$ has a cycle then so is $\mathsf{TS}_2(D_{1,n,s})$.
	Therefore, it suffices to show that $\mathsf{TS}_k(D_{1,2k-2,s})$ is acyclic for $k \in \{2,3\}$.
	Indeed, based on the number of tokens placed on the path $u_1p_1\dots p_n$ (which is at most three), one can verify that each component of $\mathsf{TS}_k(D_{1,2k-2,s})$ is either an isolated vertex, a path, or a star.
	
	\item[(c)] Observe that if $n \geq 2k - 2$, $D_{r,n,s}$ contains the independent sets $I + u_1 + v_1$, $I + u_1 + p_n$, $I + u_1 + v_s$, $I + p_1 + v_1$, $I + p_1 + v_s$, $I + u_r + v_1$, $I + u_r + p_n$, and $I + u_r + v_s$, where $I = \emptyset$ when $n = 2$ and otherwise $I$ is an independent set of the path $p_2\dots p_{n-1}$ of size $k - 2$.
	(Note that $p_2\dots p_{n-1}$ has an independent set of size at most $\lceil (n-2)/2 \rceil \geq k - 2$.)
	They indeed induce a $C_8$ in $\mathsf{TS}_k(D_{r,n,s})$.
	
	On the other hand, if $n \leq 2k - 3$ for $k \in \{2,3\}$, using a similar case-analysis as in (b), one can verify that each component of $\mathsf{TS}_k(D_{r,n,s})$ is either an isolated vertex, a path, or a star, and therefore it is acyclic.
\end{itemize}
\end{proof}

We are now ready to show the necessary and sufficient conditions for a tree/forest $G$ such that $\mathsf{TS}_k(G)$ is acyclic, where $k \in \{2,3\}$.

\begin{proposition}
\label{keq2}
        Let $T$ be a tree.
        Then $\mathsf{TS}_2(T)$ is acyclic if and only if $T$ is $\{2K_2, D_{2,2,2}\}$-free.
\end{proposition}

\begin{proof}

\begin{itemize}
            \item[($\Rightarrow$)]
 Suppose to the contrary that either $2K_2$ or $D_{2,2,2}$ is an induced subgraph of $T$.
 In the first case it follows from the discussion above that $\mathsf{TS}_2(T)$
contains a $C_4$ and in the second case that it
contains a $C_8$.

                \item[($\Leftarrow$)]
We assume that $\mathsf{TS}_2(T)$ contains a cycle and show that it must contain
one of the two forbidden subgraphs.
Firstly, suppose that $T$ is a path $P_n$.
Since $\mathsf{TS}_2(T)$ contains a cycle,
it follows from Lemma~\ref{db}(b) that $n \ge 5$ and so $T$ contains an induced $2K_2$.

We now assume $T$ has a vertex of at least degree 3.
We will construct a copy $T^\prime$ of $T$ by initially choosing a vertex $a$ of 
maximum degree in $T$ and letting $T^\prime = N[a]$. Note that $\mathsf{TS}_2(T^\prime)$ is acyclic.
We add edges from $T$ to $T^\prime$ and show after each addition
that either $T^\prime$ contains a forbidden
subgraph, so we are done, or that $\mathsf{TS}_2(T^\prime)$ remains acyclic so that $T \neq T^\prime$.

Let $b$ be a child of $a$ of highest degree,
$c$ be a child of next highest degree,
and $d$ be any other child. 
Since $\mathsf{TS}_2(T^\prime)$ is acyclic $T \neq T^\prime$ and
$b$ must have $r \ge 1$ children. 
Let $e$ be a child of $b$ with maximum degree. We add $N[b]$ to $T^\prime$
obtaining a 
copy of $D_{r,2,s}$, where $ s=\deg_T(a)-1 \ge 2$.
If $r \ge 2$, we have the required forbidden induced subgraph.
If $r = 1$ then 
by Lemma~\ref{db}(b) 
$\mathsf{TS}_2(T^\prime)$ is acyclic, so there must be extra edges to add to $T^\prime$.
If $c$ has a child $y$ then $\{b,c,e,y\}$ induce a $2K_2$.
Otherwise, $e$ must have at least one child $g$. 
Adding $eg$ to $T^\prime$ we obtain $2K_2$ as an induced subgraph on $\{a,d,e,g\}$.
This completes the proof.
\end{itemize}
\begin{figure}[htpb]
	\centering
	\includegraphics[width=0.6\textwidth]{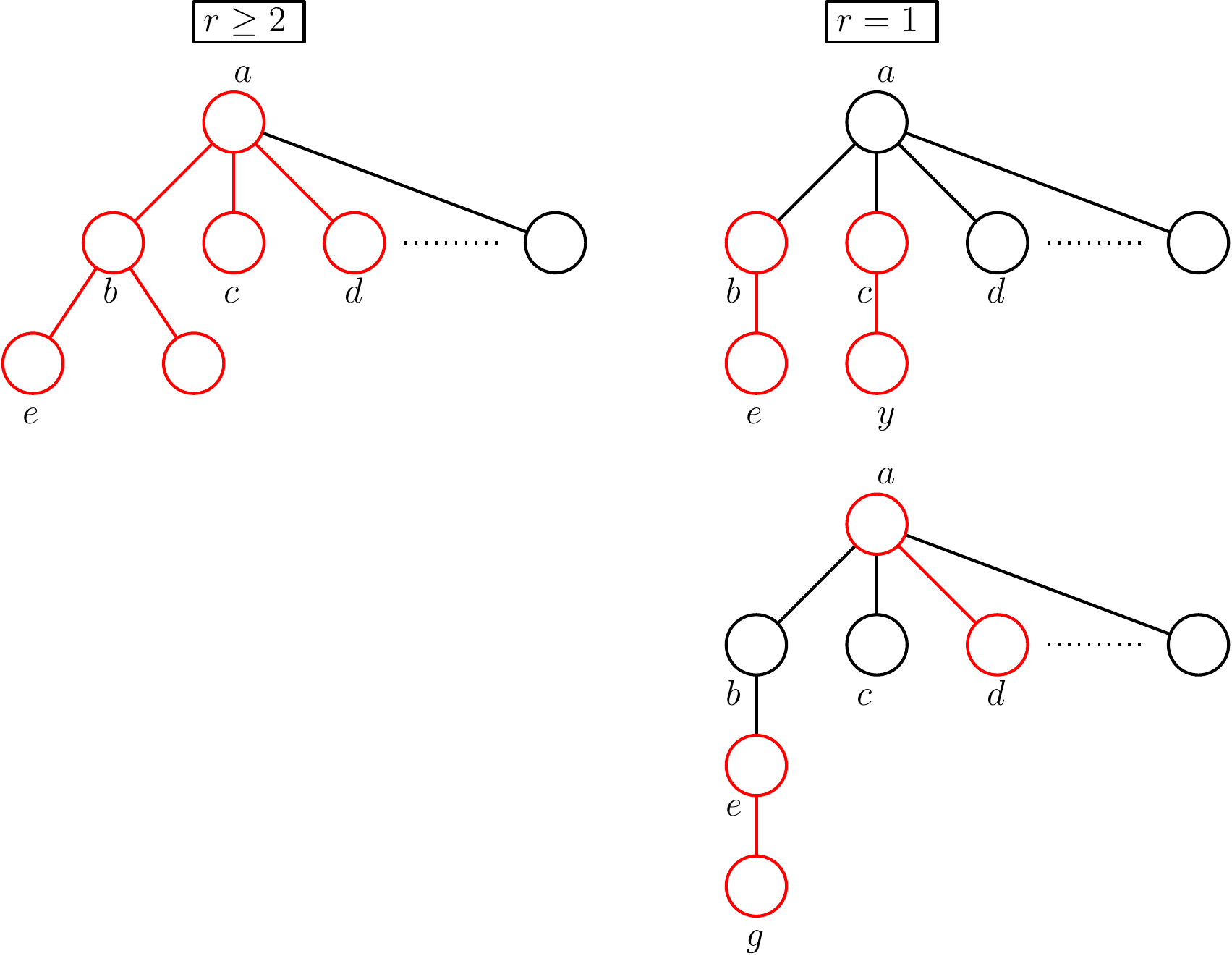}
	\caption{Illustration for Proposition~\ref{keq2}: Some trees $T^\prime$ containing $N[b]$ whose $\mathsf{TS}_2$-graphs have a cycle. Here $r$ is the number of children of $b$. Copies of $2K_2$ and $D_{2,2,2}$ are marked by red color.}
\end{figure}
\end{proof}

\begin{corollary}
Let $T$ be a tree.
Then $\mathsf{TS}_2(T)$ is acyclic if and only if $T$ is either 
$K_{1,s}$
 or $D_{1,2,s}$ for some positive integer $s$. 
\end{corollary}
\begin{proof}
The proof of Proposition~\ref{keq2} can be viewed as an algorithm that takes a tree $T$
and either terminates with $T=T^\prime$ being one of the trees in the corollary or finds
a forbidden induced graph in $T$.
\end{proof}

\begin{corollary}
\label{F2}
Let $F$ be a forest.
Then $\mathsf{TS}_2(F)$ is a acyclic if and only if $F$ is $\{2K_2, D_{2,2,2}\}$-free.
\end{corollary}
\begin{proof}
We prove that $\mathsf{TS}_2(F)$ contains a cycle if and only if $F$ contains one of the graphs in $\{2K_2, D_{2,2,2}\}$ as an induced subgraph.

Suppose that $\mathsf{TS}_2(F)$ contains a cycle. Since the independent sets have size two,
both vertices of each independent set must lie in the same connected component $T$ of $F$. By Proposition \ref{keq2}, the tree $T$ must have either $2K_2$ or $D_{2,2,2}$ as an induced subgraph.

Conversely if $F$ contains $2K_2$ or $D_{2,2,2}$ as an induced subgraph then $\mathsf{TS}_2(F)$
contains respectively a $C_4$ or a $C_8$.
\end{proof}

Moving to the case of stable sets of size three, the conditions for trees and forests differ slightly. We deal with the tree case first.
\begin{proposition}
\label{keq3}
	Let $T$ be a tree.
	Then $\mathsf{TS}_3(T)$ is acyclic if and only if $T$ is $\{2K_2+K_1, D_{2,4,2}\}$-free.
\end{proposition}
\begin{proof}
The structure of the proof is the same as for Proposition~\ref{keq2}. However, there
are more cases to consider.
\begin{itemize}
                \item[($\Rightarrow$)] 
                Suppose to the contrary that either $2K_2+K_1$ or $D_{2,4,2}$ is an induced subgraph of $T$.
                In the first case it follows that $\mathsf{TS}_3(T)$ contains a $C_4$ and in the second case that it
                contains a $C_8$. 

                \item[($\Leftarrow$)]
We assume that $\mathsf{TS}_3(T)$ contains a cycle and show that it must contain
one of the two forbidden subgraphs.
The first part of the proof is essentially the same as for Proposition~\ref{keq2}
with minor modifications.
Firstly suppose that $T$ is a path $P_n$.
Since $\mathsf{TS}_3(T)$ contains a cycle
it follows from Lemma~\ref{db}(b) that $n \ge 7$ and so $T$ contains an induced $2K_2+K_1$.

We now assume $T$ has a vertex of at least degree 3.
We will construct a copy $T^\prime$ of $T$ by initially choosing a vertex $a$ of
maximum degree in $T$ and letting $T^\prime = N[a]$. Note that $\mathsf{TS}_3(T^\prime)$ is acyclic.
We add edges from $T$ to $T^\prime$ showing after each addition
that either $T^\prime$ contains a forbidden
subgraph, so we are done, or that $\mathsf{TS}_3(T^\prime)$ remains acyclic so that $T \neq T^\prime$.

Let $b$ be a child of $a$ of highest degree,
$c$ be a child of next highest degree,
and $d$ be any other child. 
Since $\mathsf{TS}_3(T^\prime)$ is acyclic $T \neq T^\prime$ and
$b$ must have $r \ge 1$ children.           
Let $e$ be a child of $b$ with maximum degree. 
If $c$ has a child $y$ then $\{b,c,d,e,y\}$ induce a $2K_2+K_1$ and we are done.
Otherwise we add $N[b]$ to $T^\prime$
obtaining a
copy of $D_{r,2,s}$, where $s=\deg_T(a)-1 \ge 2$.
By Lemma~\ref{db}(c), $\mathsf{TS}_3(T^\prime)$ is acyclic and so $T \neq T^\prime$.
There are two cases:

\begin{itemize}
\item[($r \ge 2$)]
Let $f$ be a second child of $b$ and let $g$ be a child of $e$. 
Adding $eg$ to $T^\prime$ we obtain $2K_2+K_1$ as an induced subgraph on $\{a,d,e,f,g\}$.

\item[($r = 1$)]
Since $e$ is the only child of $b$ it must have children. Let $t \ge 1$ be the
number of children of $e$ and let $h$
be the child of $e$ of maximum degree. We add $N[e]$ to $T^\prime$ obtaining a copy
of $D_{t,3,s}$ and $\mathsf{TS}_3(T^\prime)$ is acyclic by  Lemma~\ref{db}(c).
There are two subcases:

\begin{itemize}
\item[($t \ge 2$)]
Let $i$ be any other child of $e$. 
Since $\mathsf{TS}_3(T^\prime)$ is acyclic $h$ must have at least one child $j$.
We have now constructed an induced $2K_2+K_1$ on
$\{a,d,h,i,j\}$.

\item[($t = 1$)]
If $h$ has a single child $k$ add $hk$ to $T^\prime$ which is a copy
of $D_{1,4,s}$ and again by Lemma~\ref{db}(c) $\mathsf{TS}_3(T^\prime)$ is acyclic.
So $k$ has a child $l$. Adding $kl$ to $T^\prime$ it contains an induced $P_7$ and we find the forbidden subgraph $2K_2+K_1$
on vertices $\{a,d,e,k,l\}$.
Otherwise, $h$ has at least two children including vertices $k$ and $m$. 
Adding edges $hk$ and $hm$ to $T^\prime$ we obtain
the forbidden subgraph $D_{2,4,2}$. This completes the proof.
\end{itemize}
\end{itemize}
\end{itemize}

\begin{figure}[ht]
	\centering
	\includegraphics[width=0.6\textwidth]{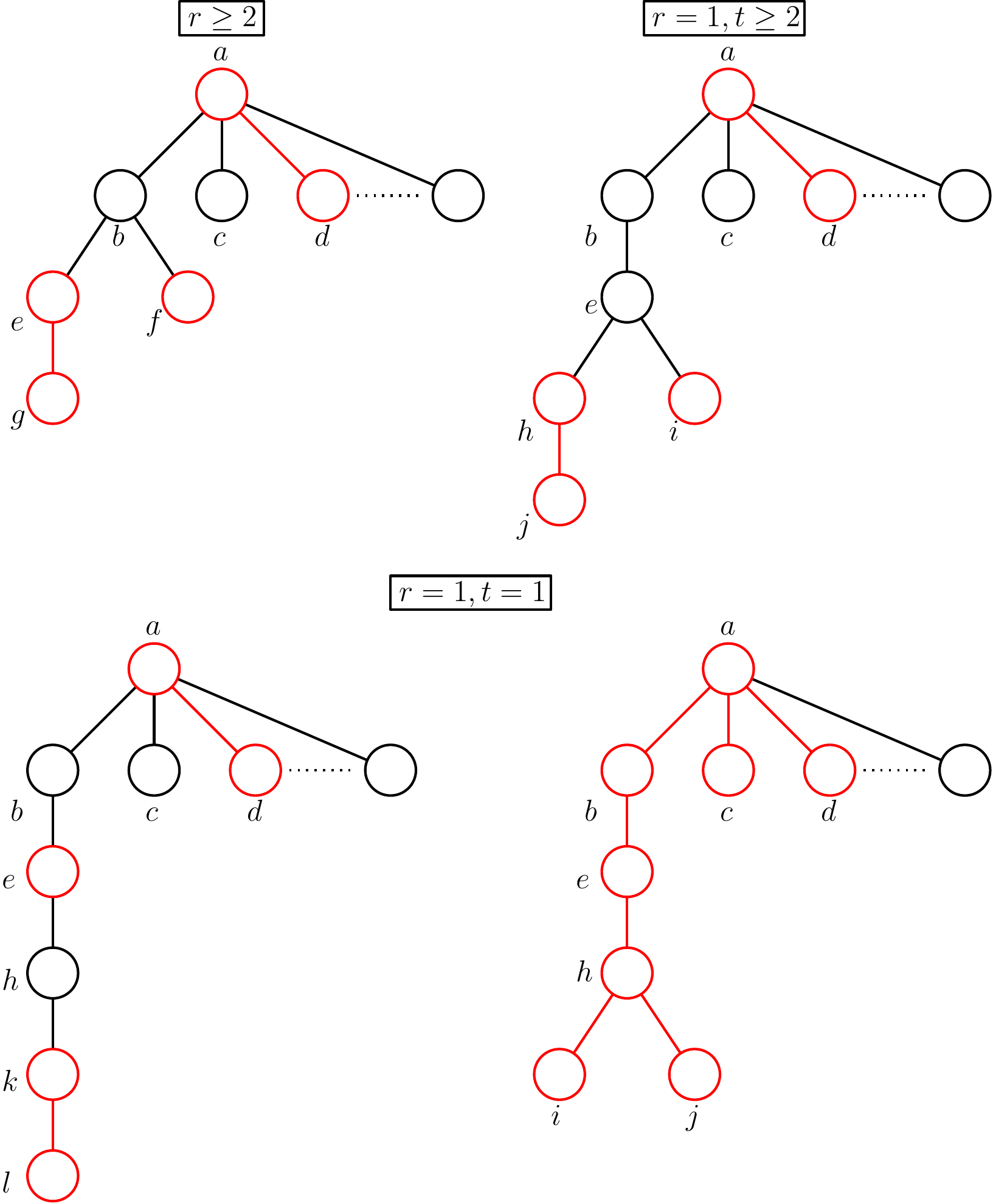}
	\caption{Illustration for Proposition~\ref{keq3}: Some trees $T^\prime$ containing $N[b]$ whose $\mathsf{TS}_3$-graphs have a cycle. Here $r$ and $t$ are respectively the number of children of $b$ and its child $e$. Copies of $2K_2 + K_1$ and $D_{2,4,2}$ are marked by red color.}
\end{figure}

\end{proof}
\begin{corollary}
Let $T$ be a tree.
Then $\mathsf{TS}_3(T)$ is a acyclic if and only if for some positive integer $s$, $T$ is either
$K_{1,s}$,
$D_{1,n,s}$ where $n \le 4$, or
$D_{r,n,s}$ where $r \ge 2$ and $n \le 3$.
\end{corollary}
\begin{proof}
The proof of Proposition~\ref{keq3} can be viewed as an algorithm that takes a tree $T$
and either terminates with $T=T^\prime$ being one of the trees in the corollary or finds
a forbidden induced graph in $T$ showing that $\mathsf{TS}_3(T)$ has a cycle.
\end{proof}

\begin{corollary}
\label{F3}
	Let $F$ be a forest.
	Then $\mathsf{TS}_3(F)$ is a forest if and only if $F$ is $\{2K_2+K_1, D_{2,2,2}+K_1, D_{2,4,2}\}$-free.
\end{corollary}
\begin{proof}
We prove that $\mathsf{TS}_3(F)$ contains a cycle if and only if $F$ contains one of the graphs in $\{2K_2+K_1, D_{2,2,2}+K_1, D_{2,4,2}\}$ as an induced subgraph.

Suppose that $\mathsf{TS}_3(F)$ contains a cycle $C$. Since the independent sets have size three,
there are three cases to consider. Firstly, if the three 
vertices of each independent set in $C$ lie in the same connected component $T$ of $F$, by Proposition \ref{keq3}, the tree $T$ must have either $2K_2+K_1$ or $D_{2,4,2}$ as an induced subgraph. Secondly, suppose two of the vertices of each stable set lie in the same connected component $T$ of $F$, which must have at least two connected components. Thus, $C$ induces a cycle in $\mathsf{TS}_2(T)$.
So by Proposition \ref{keq2}, the tree $T$ must have either $2K_2$ or $D_{2,2,2}$ as an induced subgraph. Since $F$ has at least two components, $F$ contains $2K_2+K_1$ or $D_{2,2,2}+K_1$.
Finally, suppose each vertex of each stable set lies in a different component of $F$, which therefore
has at least three components. 
At least two of these
components must be non-trivial, i.e., contain an
edge. Therefore, $F$ contains an induced $2K_2+K_1$.

Conversely, suppose $F$ contains $2K_2+K_1$, $D_{2,2,2}+K_1$ or $D_{2,4,2}$ as an induced subgraph. Then $\mathsf{TS}_3(F)$
contains a $C_4$ in the first instance or a $C_8$ in the other two.
 \end{proof}
For $k \geq 4$, we have the following proposition.
\begin{proposition}\label{prop:acyclic-TS4}
	Let $F$ be a forest.
	For $k \geq 4$, if $F$ contains either $2K_2+(k-2)K_1$, or $D_{2,2,2}+(k-2)K_1$, or $D_{2,4,2}+(k-3)K_1$ as an induced subgraph, $\mathsf{TS}_k(F)$ has a cycle.
\end{proposition}
\begin{proof}
	One can verify that $\mathsf{TS}_2(2K_2)$ contains a $C_4$, and $\mathsf{TS}_2(D_{2,2,2})$ and $\mathsf{TS}_3(D_{2,4,2})$ both contain a $C_8$.
	As a result, so do $\mathsf{TS}_k(2K_2 + (k-2)K_1)$, $\mathsf{TS}_k(D_{2,2,2} + (k-2)K_1)$, and $\mathsf{TS}_k(D_{2,4,2} + (k-3)K_1)$, respectively.
	Consequently, $\mathsf{TS}_k(F)$ has a cycle, as desired.
\end{proof}

We conclude this section with the following conjecture for $k \geq 4$.
\begin{conjecture}
	Let $F$ be a forest.
	For $k \geq 4$, if $\mathsf{TS}_k(F)$ is a forest, $F$ is $\{2K_2+(k-2)K_1, D_{2,2,2}+(k-2)K_1, D_{2,4,2}+(k-3)K_1\}$-free.
\end{conjecture}

\section{$H$-join and $H$-decomposition}
\label{sec:H-join}

Before considering (Q2), in this section, we describe an operation for combining $\mathsf{TS}_k$-graphs to produce new ones. 
We first define a family of \emph{base graphs} as follows. 
Let $V$ be a set of $k+1$ vertices including two labelled $u$ and $v$.
Then $B_k(V,uv)$ is the graph with vertex set $V$ and single edge $uv$.
We have $\mathsf{TS}_k(B_k(V,uv))=K_2$ whose two vertices are labelled by
the independent sets
$V - u$ and $V-v$. 
Next, we define the \emph{$H$-join} operation and its inverse.
\begin{definition}
	Vertex-labelled graphs $G_1$ and $G_2$ 
	are \emph{$H$-consistent} if the (possibly empty) intersection of their vertex sets
	define the same (possibly empty) common induced subgraph $H$. 
	The \emph{$H$-join} of $H$-consistent graphs $G_1$ and $G_2$ is 
	the graph $H(G_1,G_2)$ with $V(H(G_1,G_2)) = V(G_1) \cup V(G_2)$. The edges $E(H(G_1,G_2))$
	consist of $E(G_1) \cup E(G_2)$
	plus all edges $vw$ with $v \in V(G_1) \setminus V(H)$ and $w \in V(G_2) \setminus V(H)$.
\end{definition}

Recall that a \emph{(vertex) cut-set} in
a connected graph $G$ is a vertex set $W$ such that $G - W$ is disconnected.
We extend this definition to the case where $G$ is disconnected by allowing $W=\emptyset$.
We say that $W$ decomposes $G$ into two (not necessarily connected) induced subgraphs
$G_1$ and $G_2$ for which $V(G_1) \cap V(G_2) = W$
and  $V(G_1) \cup V(G_2) = V(G)$. If $G-W$ has more than two (connected) components, the decomposition
is not unique.
\begin{definition}
Let $G$ be a vertex-labelled graph.
Let $W \subset V(\overline{G}) = V(G)$ decompose the complement $\overline{G}$
into $\overline{G_1}$ and $\overline{G_2}$.
Let $H$ be the subgraph of $G$ induced by $W$.
We say that $G$ can be \emph{$H$-decomposed} into $G_1$ and $G_2$.
\end{definition}
It follows from the definitions that if $G=H(G_1,G_2)$  then
$G$ can be $H$-decomposed into $G_1$ and $G_2$, and vice versa.
It is easy to verify that the size-$k$ independent sets of $H(G_1,G_2)$ are the union of those
of $G_1$ and those of $G_2$.

As an example consider the two $4$-vertex graphs $G_1$ and $G_2$ that are paths with edge sets
$E(G_1)=\{ad, bc, cd\}$ and $E(G_2)=\{ad, ae, eb\}$. These share a common induced subgraph $H$
with $V(H)=\{a,b,d\}$ and $E(H)=\{ad\}$. We have $V(H(G_1,G_2))=\{a,b,c,d,e\}$ 
and $E(H(G_1,G_2))=\{ad,ae,bc,cd,ce,be\}$. Note that $\mathsf{TS}_2(G_1)$
is the path with edges $\{ac-ab,ab-bd\}$ and that $\mathsf{TS}_2(G_2)$
is the path with edges $\{ab-bd,bd-de\}$. It can be verified that $\mathsf{TS}_2(H(G_1,G_2))$
is the path with edges $\{ac-ab,ab-bd,bd-de\}$
which is the union of two paths $\mathsf{TS}_2(G_1)$ and $\mathsf{TS}_2(G_2)$.
(See \figurename~\ref{fig:G1G2}.)

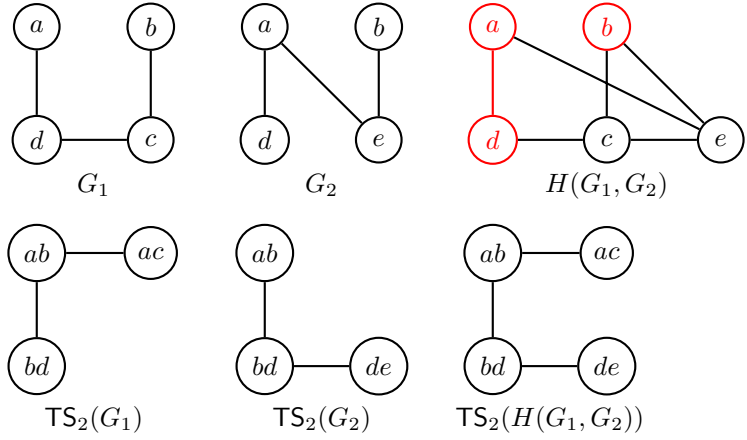
\begin{figure}[!ht]
	\centering
	\begin{adjustbox}{max width=\textwidth}
		\begin{tikzpicture}[every node/.style={circle, draw, thick, minimum size=6mm}]
			\begin{scope}
				\foreach \i/\x/\y in {a/0/0,b/1.5/0,c/1.5/-1.5,d/0/-1.5} {
					\node (\i) at (\x, \y) {$\i$};
				}
				\draw[thick] (a) -- (d) (d) -- (c) (c) -- (b);
				\node[rectangle, draw=none, fill=none] (G1) at (0.75,-2.1) {$G_1$};
			\end{scope}
			\begin{scope}[shift={(3,0)}]
				\foreach \i/\x/\y in {a/0/0,b/1.5/0,e/1.5/-1.5,d/0/-1.5} {
					\node (\i) at (\x, \y) {$\i$};
				}
				\draw[thick] (a) -- (d) (a) -- (e) (b) -- (e);
				\node[rectangle, draw=none, fill=none] (G2) at (0.75,-2.1) {$G_2$};
			\end{scope}
			\begin{scope}[shift={(6,0)}]
				\foreach \i/\x/\y in {c/1.5/-1.5,e/3/-1.5} {
					\node (\i) at (\x, \y) {$\i$};
				}
				\foreach \i/\x/\y in {a/0/0,b/1.5/0,d/0/-1.5} {
					\node[red] (\i) at (\x, \y) {$\i$};
				}
				\draw[thick] (a) edge[red] (d) (d) -- (c) (c) -- (b) (a) -- (e) (b) -- (e) (c) -- (e);
				\node[rectangle, draw=none, fill=none] (HG1G2) at (1.5,-2.1) {$H(G_1, G_2)$};
			\end{scope}
			\begin{scope}[shift={(0,-3)}]
				\foreach \i/\x/\y in {ab/0/0,ac/1.5/0,bd/0/-1.5} {
					\node (\i) at (\x, \y) {$\i$};
				}
				\draw[thick] (ab) -- (ac) (ab) -- (bd);
				\node[rectangle, draw=none, fill=none] (TS2G1) at (0.75,-2.2) {$\mathsf{TS}_2(G_1)$};
			\end{scope}
			\begin{scope}[shift={(3,-3)}]
				\foreach \i/\x/\y in {ab/0/0,de/1.5/-1.5,bd/0/-1.5} {
					\node (\i) at (\x, \y) {$\i$};
				}
				\draw[thick] (ab) -- (bd) (de) -- (bd);
				\node[rectangle, draw=none, fill=none] (TS2G2) at (0.75,-2.2) {$\mathsf{TS}_2(G_2)$};
			\end{scope}
			\begin{scope}[shift={(6,-3)}]
				\foreach \i/\x/\y in {ab/0/0,ac/1.5/0,de/1.5/-1.5,bd/0/-1.5} {
					\node (\i) at (\x, \y) {$\i$};
				}
				\draw[thick] (ab) -- (bd) (ab) -- (ac) (de) -- (bd);
				\node[rectangle, draw=none, fill=none] (TS2HG1G2) at (0.75,-2.2) {$\mathsf{TS}_2(H(G_1, G_2))$};
			\end{scope}
		\end{tikzpicture}
	\end{adjustbox}
	\caption{The graphs $G_1$, $G_2$, $H(G_1, G_2)$, and their corresponding $\mathsf{TS}_2$-graphs. Here $\mathsf{TS}_2(H(G_1, G_2)) = \mathsf{TS}_2(G_1) \cup \mathsf{TS}_2(G_2)$.}
	\label{fig:G1G2}
\end{figure}

Now consider the graph $G_3$ which is the path with edges $\{ad,cd,ce\}$.
$G_1$ and $G_3$ share a common induced subgraph $H$
with $V(H)=\{a,c,d\}$ and $E(H)=\{ad,cd\}$. We have 
$E(H(G_1,G_3))=\{ad,bc,be,cd,ce\}$. Note that 
$\mathsf{TS}_2(G_3)$
is the path with edges $\{ac-ae,ae-de\}$. In this case, $\mathsf{TS}_2(H(G_1,G_3))$
is %
the graph with edges $\{ab-ac,ac-ae,ae-de,de-bd,bd-ab,ab-ae\}$
which is the union of $\mathsf{TS}_2(G_1)$, $\mathsf{TS}_2(G_3)$, and the two additional
edges $de-bd, ab-ae$.
	(See \figurename~\ref{fig:G1G3}.)

\begin{figure}[!ht]
	\centering
	\begin{adjustbox}{max width=\textwidth}
		\begin{tikzpicture}[every node/.style={circle, draw, thick, minimum size=6mm}]
			\begin{scope}
				\foreach \i/\x/\y in {a/0/0,b/1.5/0,c/1.5/-1.5,d/0/-1.5} {
					\node (\i) at (\x, \y) {$\i$};
				}
				\draw[thick] (a) -- (d) (d) -- (c) (c) -- (b);
				\node[rectangle, draw=none, fill=none] (G1) at (0.75,-2.1) {$G_1$};
			\end{scope}
			\begin{scope}[shift={(3,0)}]
				\foreach \i/\x/\y in {a/0/0,e/1.5/0,c/1.5/-1.5,d/0/-1.5} {
					\node (\i) at (\x, \y) {$\i$};
				}
				\draw[thick] (a) -- (d) (d) -- (c) (c) -- (e);
				\node[rectangle, draw=none, fill=none] (G3) at (0.75,-2.1) {$G_3$};
			\end{scope}
			\begin{scope}[shift={(6,0)}]
				\foreach \i/\x/\y in {b/1.5/0,e/3/0} {
					\node (\i) at (\x, \y) {$\i$};
				}
				\foreach \i/\x/\y in {a/0/0,c/1.5/-1.5,d/0/-1.5} {
					\node[red] (\i) at (\x, \y) {$\i$};
				}
				\draw[thick] (a) edge[red] (d) (d) edge[red] (c) (c) -- (b) (c) -- (e) (b) -- (e) (c) -- (e);
				\node[rectangle, draw=none, fill=none] (HG1G3) at (1.5,-2.1) {$H(G_1, G_3)$};
			\end{scope}
			\begin{scope}[shift={(0,-3)}]
				\foreach \i/\x/\y in {ab/0/0,ac/1.5/0,bd/0/-1.5} {
					\node (\i) at (\x, \y) {$\i$};
				}
				\draw[thick] (ab) -- (ac) (ab) -- (bd);
				\node[rectangle, draw=none, fill=none] (TS2G1) at (0.75,-2.2) {$\mathsf{TS}_2(G_1)$};
			\end{scope}
			\begin{scope}[shift={(3,-3)}]
				\foreach \i/\x/\y in {ac/0/0,ae/1.5/0,ed/0/-1.5} {
					\node (\i) at (\x, \y) {$\i$};
				}
				\draw[thick] (ae) -- (ac) (ae) -- (ed);
				\node[rectangle, draw=none, fill=none] (TS2G3) at (0.75,-2.2) {$\mathsf{TS}_2(G_3)$};
			\end{scope}
			\begin{scope}[shift={(6,-3)}]
				\foreach \i/\x/\y in {ab/0/0,ac/1.5/0,ae/3/0,bd/0/-1.5,ed/1.5/-1.5} {
					\node (\i) at (\x, \y) {$\i$};
				}
				\draw[thick] (ab) -- (ac) (ab) -- (bd) (ae) -- (ac) (ae) -- (ed) (bd) -- (ed) (ab) edge[bend right=30] (ae);
				\node[rectangle, draw=none, fill=none] (TS2HG1G3) at (1.5,-2.2) {$\mathsf{TS}_2(H(G_1, G_3))$};
			\end{scope}
		\end{tikzpicture}
	\end{adjustbox}
	\caption{The graphs $G_1$, $G_3$, $H(G_1, G_3)$, and their corresponding $\mathsf{TS}_2$-graphs. Here $\mathsf{TS}_2(H(G_1, G_3)) \neq \mathsf{TS}_2(G_1) \cup \mathsf{TS}_2(G_3)$.}
	\label{fig:G1G3}
\end{figure}
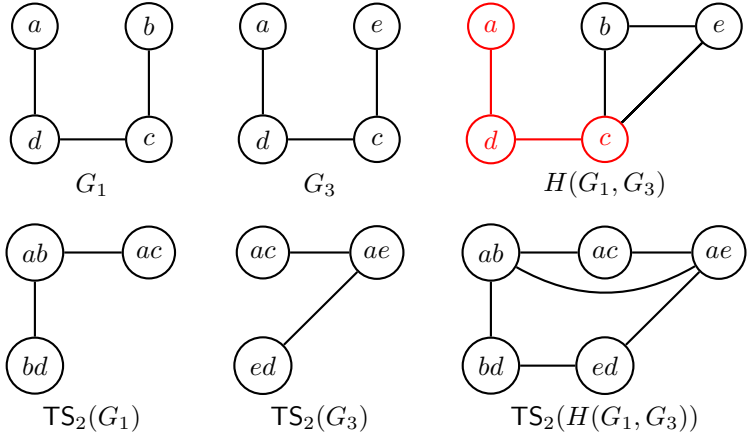

As the last example in this section, consider the graphs $G_4$ and $G_5$ as follows. $G_4$ is the cycle with edges $\{ae, eb, bc, cd, ad\}$ and $G_5$ is the graph with edges $\{ae, eb, bc, ag, eg, bg\}$.
$G_4$ and $G_5$ shares a common induced subgraph $H$ with $V(H) = \{a, e, b, c\}$ and $E(H) = \{ae, eb, bc\}$.
We have $E(H(G_4, G_5)) = \{ae, eb, bc, cd, ad, ag, eg, bg, dg\}$.
In this case, $\mathsf{TS}_2(H(G_4, G_5))$ is the (non-acyclic) graph with edges $\{ab-ac, ac-ce, ce-de, de-bd, ab-bd, ac-cg, ce-cg\}$ which is the union of $\mathsf{TS}_2(G_4)$ and $\mathsf{TS}_2(G_5)$.
(See \figurename~\ref{fig:G4G5}.)

\begin{figure}[!ht]
	\centering
	\begin{adjustbox}{max width=\textwidth}
	\begin{tikzpicture}[every node/.style={circle, draw, thick, minimum size=6mm}]
		\begin{scope}
		\foreach \i/\x/\y in {a/1/1,e/0/0,b/0/-1.3,c/2/-1.3,d/2/0} {
			\node (\i) at (\x, \y) {$\i$};
		}
		\draw[thick] (a) -- (e) -- (b) -- (c) -- (d) -- (a);
		\node[rectangle, draw=none, fill=none] (G4) at (1,-1.9) {$G_4$};
		\end{scope}
		\begin{scope}[shift={(3.5,0)}]
		\foreach \i/\x/\y in {a/1/1,e/0/0,b/0/-1.3,c/2/-1.3,g/2/0} {
			\node (\i) at (\x, \y) {$\i$};
		}
		\draw[thick] (a) -- (e) -- (b) -- (c) (a) -- (g) (e) -- (g) (b) -- (g);
		\node[rectangle, draw=none, fill=none] (G5) at (1,-1.9) {$G_5$};
		\end{scope}
		\begin{scope}[shift={(7,0)}]
		\foreach \i/\x/\y in {d/2/0,g/1/-0.3} {
			\node (\i) at (\x, \y) {$\i$};
		}
		\foreach \i/\x/\y in {a/1/1,e/0/0,b/0/-1.3,c/2/-1.3} {
			\node[red] (\i) at (\x, \y) {$\i$};
		}

		\draw[thick] (c) -- (d) -- (a) (a) -- (g) (e) -- (g) (b) -- (g) (d) -- (g);
		\draw[thick, red] (a) -- (e) -- (b) -- (c);
		\node[rectangle, draw=none, fill=none] (HG4G5) at (1,-1.9) {$H(G_4, G_5)$};
		\end{scope}
		\begin{scope}[shift={(0,-4)}]
		\foreach \i/\x/\y in {ab/1/1,ac/0/0,ce/0/-1.3,de/2/-1.3,bd/2/0} {
			\node (\i) at (\x, \y) {$\i$};
		}
		\draw[thick] (ab) -- (ac) -- (ce) -- (de) -- (bd) -- (ab);
		\node[rectangle, draw=none, fill=none] (TS2G4) at (1,-1.9) {$\mathsf{TS}_2(G_4)$};
		\end{scope}
		\begin{scope}[shift={(3.5,-4)}]
		\foreach \i/\x/\y in {ab/1/1,ac/0/0,ce/0/-1.3,cg/2/-1.3} {
			\node (\i) at (\x, \y) {$\i$};
		}
		\draw[thick] (ab) -- (ac) -- (ce) -- (cg) -- (ac);
		\node[rectangle, draw=none, fill=none] (TS2G5) at (1,-1.9) {$\mathsf{TS}_2(G_5)$};
		\end{scope}
		\begin{scope}[shift={(7,-4)}]
		\foreach \i/\x/\y in {ab/1/1,ac/0/0,ce/0/-1.3,de/2/-1.3,bd/2/0,cg/1/-0.3} {
			\node (\i) at (\x, \y) {$\i$};
		}
		\draw[thick] (ab) -- (ac) -- (ce) -- (de) -- (bd) -- (ab) (ce) -- (cg) -- (ac);
		\node[rectangle, draw=none, fill=none] (TS2G4G5) at (1,-1.9) {$\mathsf{TS}_2(H(G_4, G_5))$};
		\end{scope}

	\end{tikzpicture}
	\end{adjustbox}
	\caption{The graphs $G_4$, $G_5$, $H(G_4, G_5)$ and their corresponding (non-acyclic) $\mathsf{TS}_2$-graphs. Here $\mathsf{TS}_2(G_4, G_5) = \mathsf{TS}_2(G_4) \cup \mathsf{TS}_2(G_5)$.}
	\label{fig:G4G5}
\end{figure}
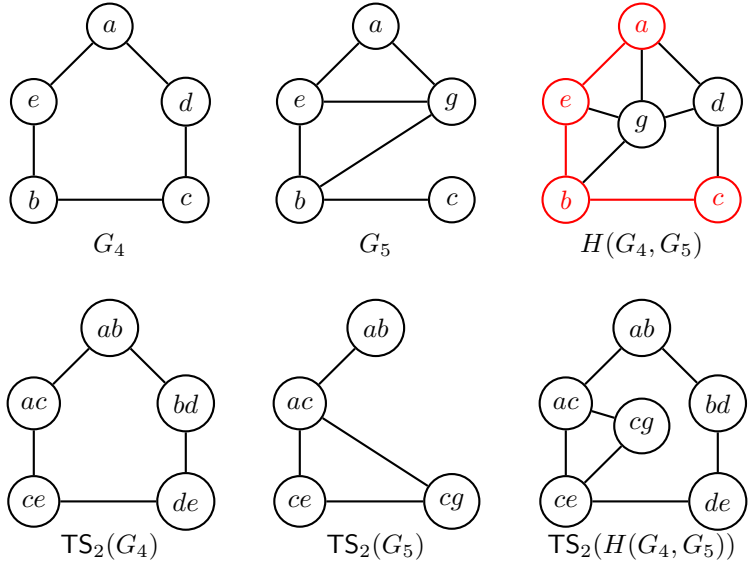

In the next proposition, we show how to compute the $\mathsf{TS}_k$-graph
of an $H$-join, generalizing the examples given above.
\begin{proposition}
	\label{overlap}
	Let $k \ge 2$ and let $G_1$ and $G_2$ be two $H$-consistent graphs.
	$\mathsf{TS}_k(H(G_1,G_2))$ is the union of $\mathsf{TS}_k(G_1)$,
 $\mathsf{TS}_k(G_2)$ and 
	for every pair of $k$-element independent sets $S_1$ in $G_1$ and $S_2$
	in $G_2$ satisfying
\begin{equation}
\label{star}
|S_1 \cap V(H) | =
     |S_2 \cap V(H) | = |S_1 \cap S_2| = k-1, 
\end{equation}
the edge between $S_1$ and $S_2$.
\end{proposition}
\begin{proof}
		
		As remarked, the $k$-element independent sets of $H(G_1,G_2)$ are the same as the union of those of $G_1$
		and $G_2$. Therefore, $V(\mathsf{TS}_k(H(G_1,G_2))) = V(\mathsf{TS}_k(G_1)) \cup V(\mathsf{TS}_k(G_2))$.
		Next, consider an edge in $E(\mathsf{TS}_k(G_1))$ (respectively, $E(\mathsf{TS}_k(G_2))$). 
		It is a token-slide between two independent
		sets $S_1$ and $S_2$ in $G_1$ (respectively, $G_2$). This remains as a token-slide in $H(G_1,G_2)$.
		Therefore, $E(\mathsf{TS}_k(G_1)) \cup E(\mathsf{TS}_k(G_2)) \subseteq E(\mathsf{TS}_k(H(G_1,G_2)))$.
		Now, consider an edge in $E(\mathsf{TS}_k(H(G_1,G_2)))$ between two independent
		sets $S_1$ and $S_2$. If both of these are independent sets are in $G_1$ (respectively, $G_2$)
		then the edge is also present in $E(\mathsf{TS}_k(G_1))$ (respectively, $E(\mathsf{TS}_k(G_2)$)).
Otherwise, we may assume the edge in $E(\mathsf{TS}_k(H(G_1,G_2)))$ has as
endpoints an independent set $S_1$ in $G_1$ (but not $G_2$) and  an independent set $S_2$ in $G_2$
(but not $G_1$).
		We have $S_1 \cap S_2 \subset V(H)$
		and since $S_1$ and $S_2$ are adjacent $|S_1 \cap S_2| =k-1$. It follows that $|S_1 \cap V(H)| = |S_2 \cap V(H)| =k-1$
and so condition (\ref{star}) is satisfied. We have shown that each edge in
$E(\mathsf{TS}_k(H(G_1,G_2)))$ is either in $\mathsf{TS}_k(G_1)$,
 $\mathsf{TS}_k(G_2)$ or satisfies condition (\ref{star}), proving the proposition.
\end{proof}

For two $H$-consistent graphs $G_1$ and $G_2$, we say that $H(G_1,G_2)$ is \emph{$k$-crossing free} if there are no $k$-element independent
sets satisfying condition (\ref{star}) of Proposition~\ref{overlap}. 
	For example, one can verify that the graphs $H(G_1, G_2)$ in \figurename~\ref{fig:G1G2} and $H(G_4, G_5)$ in \figurename~\ref{fig:G4G5} are both $k$-crossing free, while the graph $H(G_1, G_3)$ in \figurename~\ref{fig:G1G3} is not.
The following
result will be used for constructing $\mathsf{TS}_k$-trees/forests.
\begin{corollary}\label{cor:overlap}
 Let $k \ge 2$ and let $G_1$ and $G_2$ be two $H$-consistent graphs. 
$H(G_1,G_2)$ is $k$-crossing free if and only if
\begin{equation}
\label{union}
\mathsf{TS}_k(H(G_1,G_2))= \mathsf{TS}_k(G_1) \cup
 \mathsf{TS}_k(G_2). 
\end{equation}
\end{corollary}
\begin{proof}
If $H(G_1,G_2)$ is $k$-crossing free then (\ref{union}) follows
from Proposition \ref{overlap}. Otherwise their exist $k$-element independent sets
 $S_1$ is in $G_1$ and $S_2$ is in $G_2$ satisfying (\ref{star}).
This implies that $\mathsf{TS}_k(H(G_1,G_2))$ contains an additional
edge between $S_1$ and $S_2$.
\end{proof}
Therefore, if $H(G_1,G_2)$ is $k$-crossing free
and both
$\mathsf{TS}_k(G_1)$ and
 $\mathsf{TS}_k(G_2)$ are acyclic, then so is $\mathsf{TS}_k(H(G_1,G_2))$.
The reason for allowing $H$ to be empty in defining
an $H$-join is that the corollary then applies to
vertex disjoint graphs $G_1$ and $G_2$, since in this
case  $H(G_1,G_2)$ is trivially $k$-crossing free.
Therefore, we can create reconfiguration graphs 
that are forests from reconfiguration 
graphs that are trees (or forests).

The following result follows from the relationship between $H$-join
and $H$-decomposition discussed above.
\begin{corollary}
If $G$ can be $H$-decomposed into $G_1$ and $G_2$ 
and $H(G_1,G_2)$ is $k$-crossing free
then $\Tk(G)$ can be decomposed into $\mathsf{TS}_k(G_1) \cup \mathsf{TS}_k(G_2)$.
\end{corollary}

\section{Results on (Q2)}
\label{Q2}
We currently have no general necessary and sufficient conditions for
when a forest $F$ is a $\mathsf{TS}_k$-graph, but we present some partial
results in this section.
Firstly, we recall that in~\cite{AvisH22} it is shown that $P_n$
is a $\mathsf{TS}_k$-graph for all $n \ge 1$ and $k \ge 2$
and $K_{1,n}$ is a $\mathsf{TS}_k$-graph if and only if $n \le k$.
In this section, we show how to construct acyclic $\mathsf{TS}_k$-graphs 
from graphs that have a single edge using the join operation that
was introduced in Section~\ref{sec:H-join}.
We show that it gives an alternate method of constructing
$\mathsf{TS}_k$-graphs which are paths and stars.
Moreover, this operation can also be applied to construct
more general $\mathsf{TS}_k$ trees/forests, especially members of the classes $k$-ary trees and $D_{r,n,s}$.

\subsection{Paths and stars revisited}

Using just the base graphs and the $H$-join operation defined in Section~\ref{sec:H-join}, we can 
obtain large families of $\mathsf{TS}_k$ trees/forests.
We begin with paths.
For any $k \ge 2$, let
$J_k = \{b_1, \dots, b_k\}$ be an independent set of size $k$ and 
define the base graph $B_k^i = B_k(J_{k-2} \cup \{a_i,a_{i+1},a_{i+2}\},a_ia_{i+2})$
and let $G_2=B_k^i$.
\begin{proposition}
	\label{paths}
	For $ i \ge 2$,
	$G_{i}$ and $B_k^i$
	are $H$-consistent with $H$ being the independent set $J_{k-2} \cup \{a_i,a_{i+1}\}$.
	Define $G_{i+1}:=H(G_{i},B_k^i)$.
	Then 
	\[
	\mathsf{TS}_k(G_{i+1}) = \mathsf{TS}_k(G_{i}) \cup \mathsf{TS}_k(B_k^i) \simeq P_{i+1} .
	\]
\end{proposition}
\begin{proof}
We will prove by induction, for $i \ge 2$, that $\mathsf{TS}_k(G_i)$ is the
path $P_i$ with vertices labelled $J_{k-2}\cup \{a_j,a_{j+1}\}, j=1,\dots,i$.
For the base case $i=2$, we observe that indeed
	$ \mathsf{TS}_k(B_k^i)$ is a $P_2$ 
	with vertices labelled $J_{k-2} \cup \{a_1,a_2\}$ and $J_{k-2} \cup \{a_2,a_3\}$.

For the inductive step we observe that, for $i \ge 2$, 
$G_{i}$ and $B_k^i$
are $H$-consistent with $H$ the independent set $J_{k-2} \cup \{a_i,a_{i+1}\}$.
To verify that $H(G_i, B^i_k)$ is $k$-crossing free, %
note that the only independent set we need to consider in
$B_k^i$ is $J_{k-2} \cup \{a_{i+1},a_{i+2}\}$. 
In the path $P_i$ 
which is $\mathsf{TS}_k(G_{i})$, the candidate independent sets are
$J_{k-2}\cup \{a_j,a_{j+1}\}, j=1,\dots,i$.
Their intersection with $B_k^i$ is $J_{k-2}$
which has cardinality $k-2$. Therefore condition (\ref{star}) of Proposition~\ref{overlap} is not satisfied, which indeed confirms that $H(G_i, B^i_k)$ is $k$-crossing free. %
We define $G_{i+1}:= H(G_{i},B_k^i)$.
By Corollary~\ref{cor:overlap}, %
$\mathsf{TS}_k(G_{i+1})$ is the union of the above labelled $P_i$ with
a $P_2$ with endpoints $J_{k-2} \cup \{a_i,a_{i+1}\}$ and $J_{k-2} \cup \{a_{i+1},a_{i+2}\}$.
This is the required $P_{i+1}$.
\end{proof}
An easy inductive argument based on the $H$-join in the proposition shows that,
for $i \ge 2$, $G_i$ is isomorphic to
$\overline P_{n+1} \cup J_{k-2}$, a result proved in Corollary 5(a) of \cite{AvisH22}.
(Observe that the vertex $a_{i+1}$ in $G_i$ is adjacent to every $a_j$ for $1 \leq j \leq i-1$.)

Next we consider graphs $G_i$ such that
$\mathsf{TS}_k(G_i)$ is the star
$K_{1,i}$. 
For $k \ge 2$ and $1 \le i \le k$, let
$I_k = \{a_1, \dots, a_k\}$ be an independent set of size $k$,
define the base graph $C_k^i = B_k(I_k+b_i,a_ib_i )$
and let $G_1=C_k^1$.
\begin{proposition}
        \label{stars}
        For $k \ge 2$ and $1 \le i \le k$,
        $G_{i}$ and $C_k^{i+1}$
        are $H$-consistent with $H$ being the independent set $I_k$.
        Define $G_{i+1}:=H(G_{i},C_k^{i+1})$.
        Then
        \[
        \mathsf{TS}_k(G_{i+1}) = \mathsf{TS}_k(G_{i}) \cup \mathsf{TS}_k(C_k^{i+1}) \simeq K_{1,i+1} .
        \]
\end{proposition}
\begin{proof}
We will prove by induction, for $i \ge 1$, that $\mathsf{TS}_k(G_i)$ is the
star $K_{1,i}$ with centre labelled $I_k$ and leaves labelled $I_k+b_j-a_j, j=1, \dots,i$.
For the base case $i=1$, we observe that indeed
        $ \mathsf{TS}_k(C_k^i)$ is a $K_{1,1}$
        with centre labelled $I_k$ and leaf labelled $I_k +b_1-a_1$.

For the inductive step we observe that, for $i \ge 1$,
$G_{i}$ and $C_k^{i+1}$
are $H$-consistent with $H$ the independent set $I_k$.
To verify that $H(G_i, C_k^{i+1})$ is $k$-crossing free, %
note that the only independent set we need to consider in
$C_k^{i+1}$ is $I_k + b_{i+1}-a_{i+1}$.
In the above labelled $K_{1,i}$
which is $\mathsf{TS}_k(G_{i})$, 
the candidate independent sets for condition (\ref{star}) of Proposition~\ref{overlap}
are  $I_k + b_{j}-a_j, j=1, \dots, i$. Their intersection with  $I_k + b_{i+1}-a_{i+1}$
has cardinality $k-2$. Therefore, condition (\ref{star}) is not satisfied.
We define $G_{i+1}:= H(G_{i},C_k^{i+1})$.
By Corollary~\ref{cor:overlap}, %
$\mathsf{TS}_k(G_{i+1})$ is the union of the above labelled $K_{1,i}$ and
a $K_{1,1}$ with centre also labelled $I_k$ and leaf labelled $I_k + b_{i+1}-a_{i+1}$.
This is the required $K_{1,i+1}$.
\end{proof}

\subsection{$k$-ary trees}

In this section, we show that for each $k \ge 2$, every $k$-ary tree 
is a $\Tkp$-graph (Proposition~\ref{k-ary tree}).
Next, we show that any tree $T$ is an induced subgraph of some \emph{$\mathsf{TS}_2$-forest} (Proposition~\ref{forestmerge}).
Moreover, we state and prove the necessary and sufficient conditions for $T$ to be an induced subgraph of some \emph{$\mathsf{TS}_2$-tree} (Proposition~\ref{prop:no-4tree}).
Additionally, when $T = K_{1,n}$, we describe a sufficient condition for $T$ to be an induced subgraph of some $\mathsf{TS}_k$-tree (Proposition~\ref{prop:tree-contain-K1n}).

We begin by defining a canonical vertex labelling.
In this subsection, for any integer $n$, define
$I_n := \{a_1,\dots,a_n\}$ and $J_n := \{b_1,\dots,b_n\}$.
\begin{definition}
\label{canon}
Let $k \ge 2$ and $G$ be a graph for which $T:=\Tkp(G)$ is a $k$-ary tree. 
We say that $G$ and $T$ are 
\emph{canonically labelled} if
\begin{itemize}
\item[(a)] the root of $T$ is labelled $I_{k+1}$,
\item[(b)] the $d \le k$ children of the root are labelled $I_{k+1}-a_i+b_i, i=1,\dots,d$, 
\item[(c)] the labels $b_j,~j=d+1,\dots,k$ (if any) are not used,
and 
\item[(d)] all other nodes in $T$ receive a label $S$ such that 
$|I_{k+1} \cap S| \le k-1$.
\end{itemize}
\end{definition}
It is clear that labelling $K_{1,d},~d \le k$ according to (a) and (b) with
root the centre of the star is
a canonical labelling. In this subsection, we will show that every $k$-ary tree
has canonical labelling hence proving it is a $\Tkp$-graph.
First, we give a lemma that shows how to combine
canonically labelled $k$-ary trees to get a larger $k$-ary tree that is canonically labelled.
\begin{lemma}
\label{kmerge}
For integers $k \ge 2$ and $1 \le i \le d \le k$, 
let $G_i$ be a graph for which $\Tkp(G_i)$ a canonically labelled $k$-ary tree.
We can construct a canonically labelled $k$-ary tree $T$ isomorphic to
the tree formed by 
choosing a new root and adjoining it to the root
of each $T_i$. 
\end{lemma}
\begin{proof}
The proof consists of showing that we can make a series of $H$-joins between the leaves
of a canonically labelled $K_{1,d}$ and the roots of the canonically labelled
trees $T_i, i=1,\dots,d$, after a suitable relabelling. 
Suppose the root of $T_i$ has $n_i \le k$
children.
We relabel the vertices in the underlying graphs as follows:
\begin{itemize}
\item[(i)] relabel vertices of the $G_i$ not in $I_{k+1} \cup J_k$ to be distinct, ie,
for $1 \le i \le j \le d$,
we have $V(G_i) \cap V(G_j) \subseteq I_{k+1} \cup J_k$,
\item[(ii)] for $i=1,\ldots d,~j=1, \dots,n_i$ set $b_j \leftarrow b_j^i$, where the $b_j^i$ were
previously unused, and
\item[(iii)] for $i=1,\ldots d$, set $a_i \leftarrow a_{k+1}$ and $a_{k+1} \leftarrow b_i$.
\end{itemize}
By an abuse of notation,
for simplicity we let for $i=1,\ldots,d$, $G_i$ and $T_i$ refer to the relabelled graphs
and trees.
Item (i) ensures that the only labels shared between two trees are in $I_{k+1} \cup J_k$,
(ii) ensures that all labels from $J_k$ in the $T_i$ are given unique labels to
avoid clashes, and (iii) gives the root of $T_i$ a correct label to be a child
of a new root labelled $I_k$. We note that after relabelling $b_i$ only appears
in $T_i$, $a_i$ does not appear in $T_i$ and
the only labels shared between the $T_i$ are in $I_k$.
Furthermore all tree vertices have unique labels.

Next take a canonically labelled graph $G^0$ such that $\Tkp(G^0) \simeq K_{1,d}$,
with the centre of the star labelled $I_{k+1}$.
For $i=1,\ldots, d$, we claim that the
$H$-join $G^i:=H(G^{i-1},G_i)$ is well-defined, %
$k$-crossing free,
and %
$\mathsf{TS}_{k+1}(G^i)$ is canonically labelled.
To see this, note at that iteration $i$, $V(G^{i-1}) \cap V(G_i) = I_{k+1}-a_i+b_i$
which is the label of the root of $T_i$ and a leaf of $\Tkp(G^{i-1})$.
Definition \ref{canon}(d) implies that condition (\ref{star}) of Proposition~\ref{overlap} %
is not satisfied. Therefore
by %
Corollary~\ref{cor:overlap},
$\Tkp(G^i)$ is obtained from $\Tkp(G^{i-1})$
by appending $T_i$ to the corresponding leaf in $\Tkp(G^{i-1})$.
The conditions of Definition \ref{canon} are satisfied so $\Tkp(G^i)$ is canonically labelled.
At the end of iteration $d$, $T:=\Tkp(G^d)$ is the required tree.
\end{proof}
The construction described in the proof is illustrated in \figurename~\ref{fig:kmerge}.
\begin{figure}[!ht]
\centering
\includegraphics[width=0.7\textwidth]{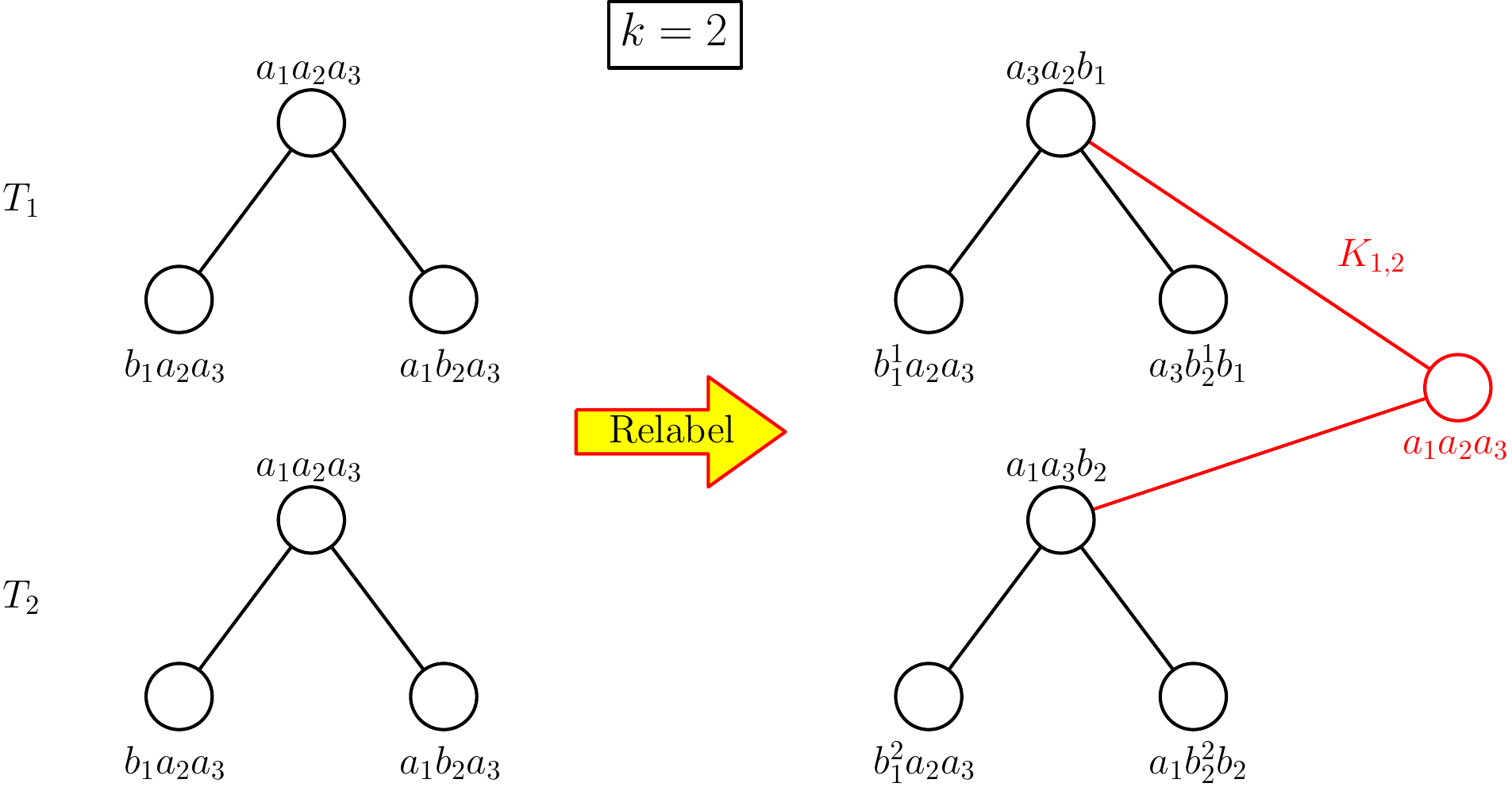}
                \caption{Construction of $D_{2,3,2}$ from two $K_{1,2}$.}
                \label{fig:kmerge}
\end{figure}
We may now prove the main result of this section.
\begin{proposition}\label{k-ary tree}
For every $k$-ary tree $T$, there is a canonically labelled graph $G$
such that $T \simeq \Tkp(G)$.
\end{proposition}
\begin{proof}
Suppose that the root $r$ of $T$ has $d \le k$ children.
We prove the proposition by induction on the height $t$ of $T$,
which is the length of the longest path to a leaf from the root.
If $t = 1$ then $T \simeq K_{1,d}$ and so has a canonically representation
as described following Definition \ref{canon}.
Otherwise, by deleting $r$ we obtain $d$ subtrees $T_i, i=1,\ldots,d$, which are
also $k$-ary trees, with height less than $t$. Therefore, by induction each
$T_i$ can be represented by a canonically labelled graph $G_i$.
It follows from Proposition \ref{kmerge} that we can perform $d$
$H$-joins to obtain a canonically labelled graph $G$ for which
$T \simeq \Tkp(G)$.
\end{proof}

As noted in Section~4 of \cite{AvisH22},
$K_{1,{k+1}}$ is an example of a $k$-ary tree that is not an $\Tk$-graph 
so the proposition is tight. 
Nevertheless, if we add a sufficient number
of isolated vertices to $K_{1,t}$, for 
$t > k$,
it becomes a $\Tw$-graph---a result we will now prove in general.
We will need a special labelling of a tree
that will be defined next.
\begin{definition}
\label{well}
A tree $T$ is
\emph{well-labelled} if
\begin{itemize}
\item[(a)] the root $r$ of $T$ is labelled $ab$,
\item[(b)] the $d$ children of $r$ have roots labelled
$r_i=bc_i, i=1,\dots,d-1$
and $r_d=ac_d$,
\item[(c)] the only labels containing $a$ and $b$ are 
$ab,ac_d,bc_i,1 \le i \le d-1$, 
and
\item[(d)] for $i=1,\ldots,d$ label $c_i$ only occurs in the subtree with root $r_i$.
\end{itemize}
\end{definition}
We note that there is nothing special about the
ordering of the subtrees of $r$. The subtree rooted at $r_i$ can
play the role of $r_d$ by relabelling those two subtrees with
the exchanges $a \leftrightarrow b$ and 
$c_i \leftrightarrow c_d$, which leaves $T$ well-labelled.
As an example, for $d \ge 1$ we can well-label $K_{1,d}$ simply by using (a) and (b).
Consider the graph $G$ defined by $V(G)=\{a,b\} \cup \{c_i:1 \le i \le d\}$
and $E(G)=\{ac_i,c_ic_d:1 \le i \le d-1\} \cup \{bc_d\}$. 
Furthermore let $J=\{c_ic_j : 1 \le i < j \le d-1\}$. Then it is not hard to 
verify that $\Tw(G) \simeq K_{1,d} + (d-1)(d-2)K_1$, where 
the $K_{1,d}$ is well-labelled and the $K_1$
are labelled by the set $J$. This motivates the following definition.
\begin{definition}
A tree $T$ is \emph{well-labelled by a labelled graph $G$} if there is an integer $n$
such that $\Tw(G) \simeq T + n K_1$ and $T$ is well-labelled.
\end{definition}
We now show the following general result.
\begin{proposition}
\label{forestmerge}
For every tree $T$ there is a 
graph $G$ and integer $n$ 
such that $T$ is well-labelled by $G$ and $\Tw(G) \simeq T + nK_1$.
\end{proposition}
\begin{proof}
The proof is by induction on $N$, the number of nodes in a given tree $T$.
As noted above, the proposition is true for all stars $K_{1,t}$
and these act as base cases.
For the inductive step, assume the proposition is true for all trees on
$N$ nodes and consider a tree $T$ with $N+1$ nodes. If $T$ is a star
we are done.
Otherwise, let $r$ be the root of $T$ and assume $r$ has degree $d$
with its children $r_i$ being roots of subtrees $T_i, 1, \dots, d$. 
We may also assume that 
$T_d$ is a subtree of $T$ with height at least one.
We now construct two trees from $T$. The first, $T^1$ consists of $T$
with subtree $T_d$ deleted and a pendant
vertex added to its root $r$. The second, $T^2$ consists of $T_d$
with a pendant vertex added to its root $r_d$.
By induction, there are integers $n_1,n_2$ and graphs $G^1,G^2$ which
well-label $T^1$ and $T^2$ such 
that $\Tw(G^1) \simeq T^1 + n_1 K_1$ and $\Tw(G^2) \simeq T^2 + n_2 K_1$.
Apart from the vertex labels used in Definition \ref{well}, we may assume
the vertex labels in $G^1$ and $G^2$ are different.

We will show that $G^1$ and a relabelled $G^2$ 
can be $H$-joined and that this will identify
the pendant edges added to $T^1$ and $T^2$ to give us back $T$.
In $T^1$ we note that root $r$ is labelled $ab$, and by
relabelling subtree roots if necessary, that the added pendent vertex 
can be labelled $ac_d$. In $T^2$ the root $r_d$ is also labelled $ab$
and we can again assume the added pendant vertex is labelled $ac_d$.
In $T^2$ we interchange the labels $b \leftrightarrow c_d$ and
set $c_i \leftarrow c_i^{\prime} ,i=1,\ldots,d-1$, for labels $c_i^{\prime}$ that are unused in
either $T^1$ or $T^2$. 
Let $G^3$ and $T^3$ denote the relabelled $G^2$ and
$T^2$.
Setting $H=\{a,b,c_d \}$,
we have $V(G^1) \cap V(G^3)= H$.
$H$ induces the same
subgraph, containing the single edge %
$bc_d$,
in both $G^1$ and $G^3$.
$G^1$ and $G^3$ are $H$-consistent and
since $k=2$ and their vertex sets
are otherwise disjoint, condition (\ref{star}) of Proposition \ref{overlap}
is not satisfied. Let $G^4=H(G^1,G^3)$.
Applying Corollary~\ref{cor:overlap} %
we have that 
\[
T^4 := \Tw(G^4) \simeq \Tw(G^1) \cup \Tw(G^3) \simeq \{ T^1 + n_1 K_1\}
 \cup \{ T^3 + n_2 K_1\} \simeq T + (n_1+n_2) K_1.
\]
is well-labelled by $G^4$.
This proves the proposition. 
\end{proof}
The proof of the proposition is illustrated in Figure \ref{fig:starmerge}.
\begin{figure}[!ht]
\centering
\includegraphics[width=0.7\textwidth]{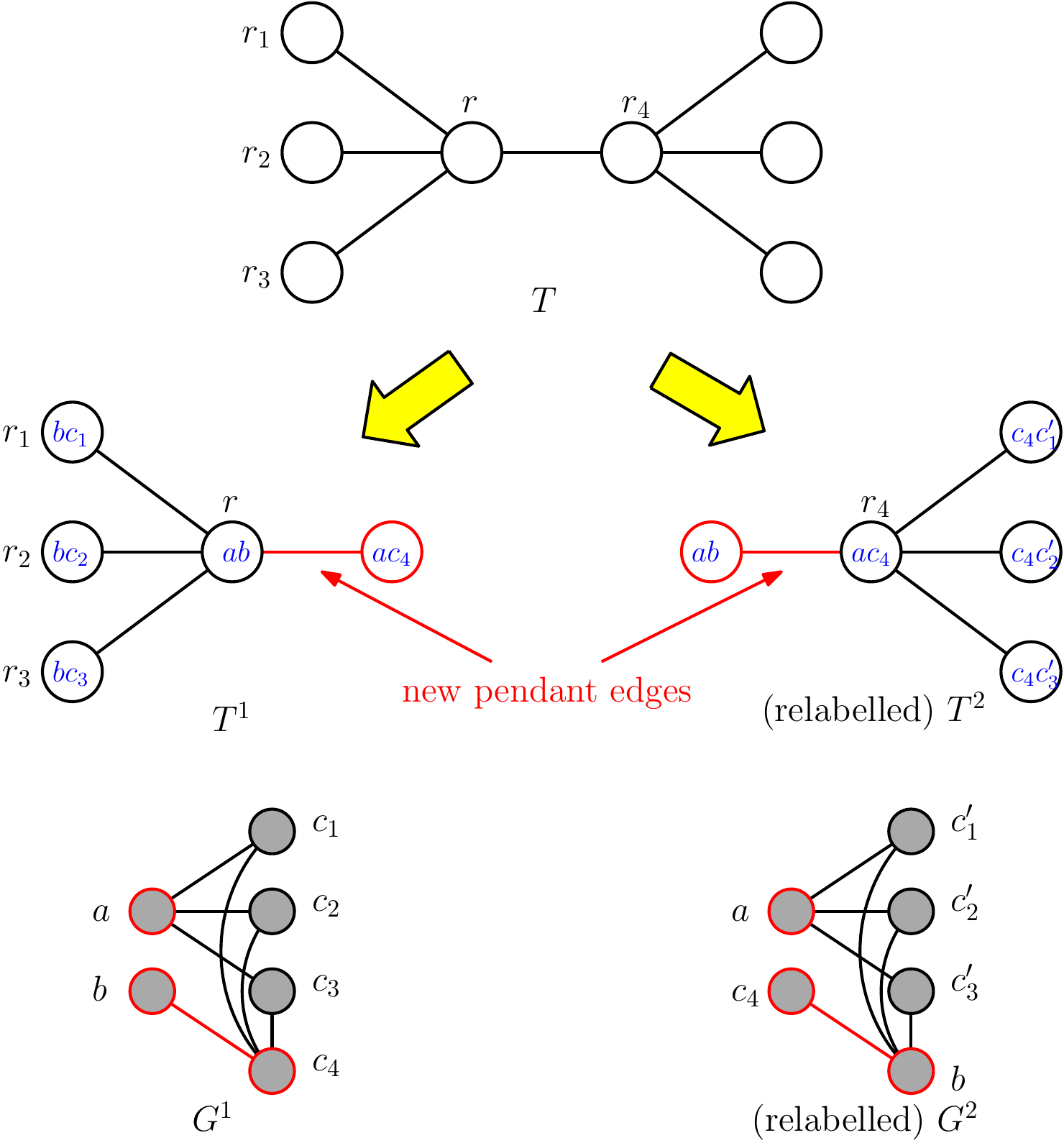}
\caption{Illustrating Proposition \ref{forestmerge}.}
\label{fig:starmerge}
\end{figure}
The proposition tells us that for every tree $T$ there is a
graph $G$ for which $\Tw(G)$ is forest containing $T$ as an
induced subgraph. 
Therefore, there can be no forbidden induced subgraph
characterization of which \emph{forests} are $\Tw$-graphs.
However, this does not imply that there can be no forbidden induced subgraph
characterization of which \emph{trees} are $\Tw$-graphs.
Indeed, in the next propositions, we present some of such characterizations.
\begin{figure}[!ht]
	\centering
	\includegraphics[width=0.7\textwidth]{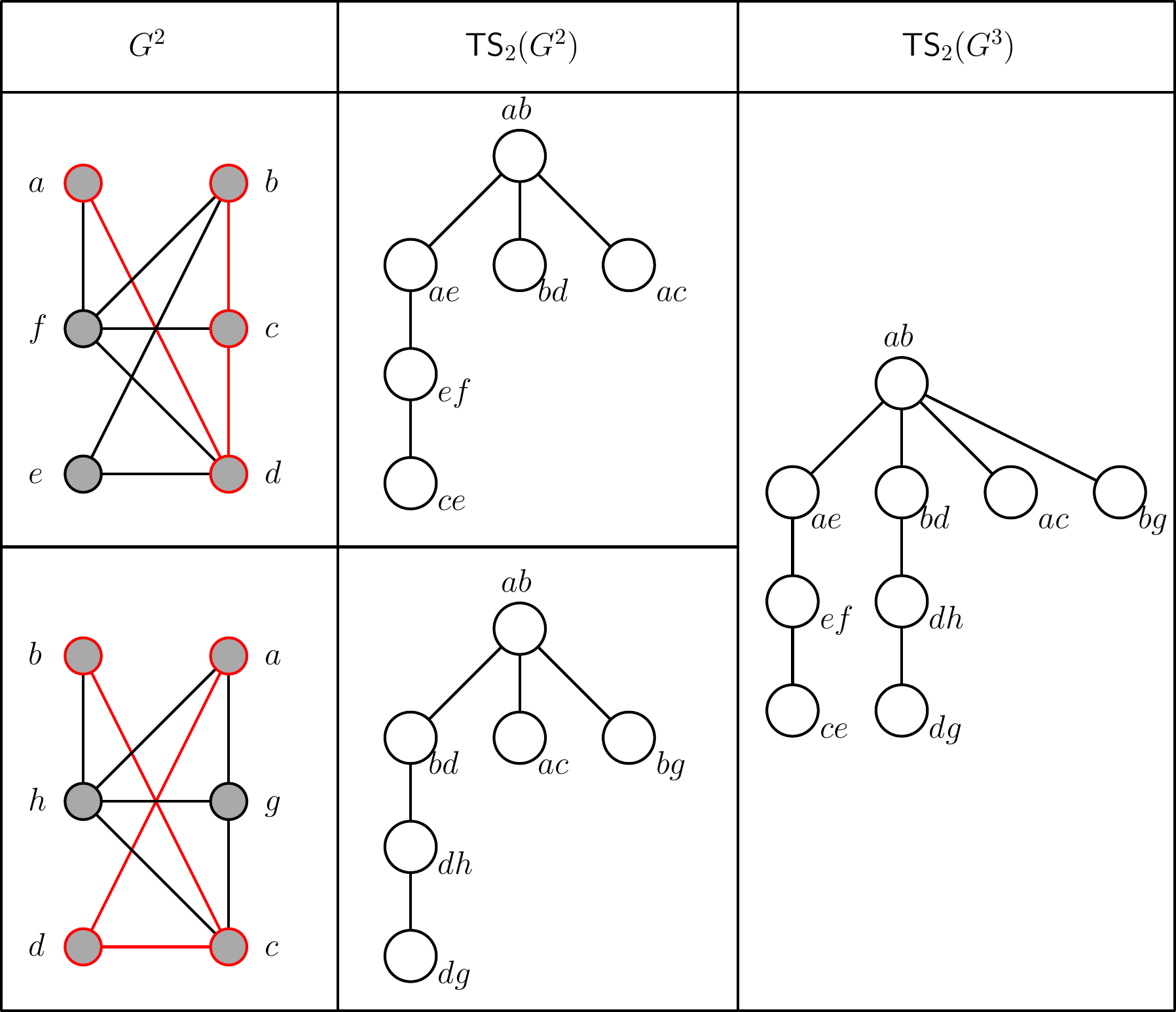}
	\caption{Taking $H$-join of two copies of $G^2$, where $H$ is the path $adcb$, results a graph $G^3$ such that $\Tw(G^3)$
		is isomorphic to a $P_7$ with two pendant vertices attached to the
		midpoint of the path.}
	\label{fig:3tree}
\end{figure}

\begin{proposition}\label{prop:no-4tree}
	Let $T$ be a tree.
	Then there exists a $\mathsf{TS}_2$-tree containing $T$ if and only if $T$ is a $3$-ary tree.
\end{proposition}
\begin{proof}
	\begin{itemize}
		\item[($\Leftarrow$)] 
			In the proof of Proposition \ref{forestmerge}, we see that isolated vertices
			are only added when the base case of a star appears as a subproblem.
			Therefore, it suffices to consider only the case $T=K_{1,t}, 1 \le t \le 4$.
			As we have noted, neither $K_{1,3}$ nor $K_{1,4}$ are $\Tw$-graphs.
			It is not hard to see that there is a $G^1$ such that 
$\Tw(G^1) \simeq K_{1,3} + K_1$. 
			However, by adding an extra vertex to $G^1$,
			we can construct a graph $G^2$ such that $\Tw(G^2) \simeq D_{1,3,2}$.
			Furthermore, we can construct a graph $G^3$ by applying $H$-join to two copies of $G^2$ with slightly different vertex-labellings such that $\Tw(G^3)$
			is isomorphic to a $P_7$ with two pendant vertices attached to the midpoint of the path. (See \figurename~\ref{fig:3tree}.)
			Thus, if follows that when $T=K_{1,t}, 1 \le t \le 4$, we can embed it as an induced subgraph of a tree $T^\prime=\Tw(G)$, for some graph $G$ (see Figure \ref{fig:3tree}).
			Our proof of the if direction is complete.
		\item[($\Rightarrow$)] We show that if $T$ is a $k$-ary tree but not a $3$-ary tree for $k \geq 4$ then there does not exist any $\mathsf{TS}_2$-tree $T^\prime$ containing $T$ (as an induced subgraph).
		(By definition, any $k$-ary tree is also a $\ell$-ary tree for $\ell \geq k$.)
		Let $x$ be a vertex of $T$ whose degree is at least five.
		(Since $T$ is a $k$-ary tree but not a $3$-ary tree, such a vertex $x$ exists.)
		
		Suppose to the contrary that $T^\prime$ exists, i.e., there exists a graph $G^\prime$ such that $T^\prime \simeq \mathsf{TS}_2(G^\prime)$ contains $T$.
		Without loss of generality, assume that $x$ is labelled by $ab$, where $\{a, b\}$ is a size-$2$ stable set of $G^\prime$.
		By the pigeonhole principle, we may further assume that three neighbors $x_1$, $x_2$, and $x_3$ of $x$ are labelled $ac$, $ad$, and $ae$, respectively.
		Since $T^\prime$ is a tree, it follows that $cd$, $ce$, and $de$ are respectively the labels of $y_1$, $y_2$, and $y_3$ where $y_i$ is not adjacent to any of $\bigcup_{j}\{x_j\} + x + \bigcup_{j \neq i}\{y_j\}$ for $1 \leq i, j \leq 3$.
		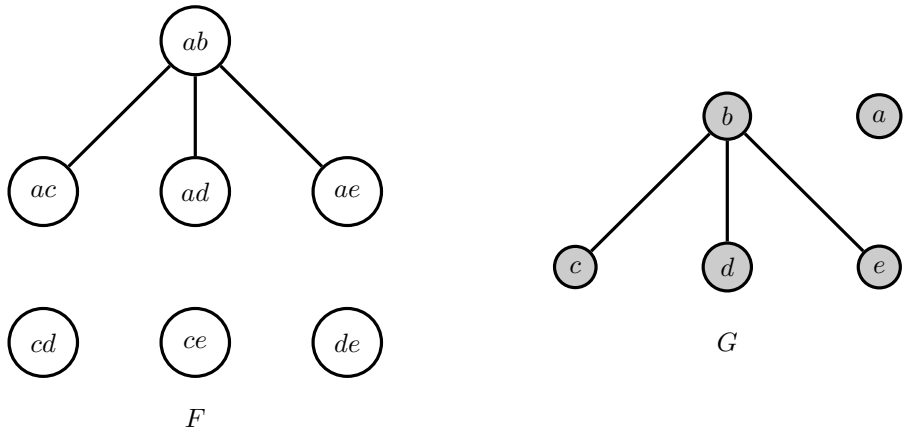
\begin{figure}[!ht]
			\centering
			\begin{adjustbox}{max width=\textwidth}
			\begin{tikzpicture}[every node/.style={draw, circle, very thick, fill=white, minimum width=9mm}]
				\begin{scope}
					\foreach \i/\x/\y in {ab/0/0,ac/-2/-2,ad/0/-2,ae/2/-2,cd/-2/-4,ce/0/-4,de/2/-4} {
						\node (\i) at (\x, \y) {$\i$};
					}
					\draw[very thick] (ab) -- (ac) (ab) -- (ad) (ab) -- (ae);
					\node[rectangle, draw=none, fill=none] at (0,-5) {$F$};
				\end{scope}
				\begin{scope}[shift={(7,-1)}]
					\foreach \i/\x/\y in {a/2/0,b/0/0,c/-2/-2,d/0/-2,e/2/-2} {
						\node[minimum width=5mm, fill=gray!40!white] (\i) at (\x, \y) {$\i$};
					}
					\draw[very thick] (b) -- (c) (b) -- (d) (b) -- (e);
					\node[rectangle, draw=none, fill=none] at (0,-3) {$G$};
				\end{scope}
			\end{tikzpicture}
			\end{adjustbox}
			\caption{The graphs $F$ and $G$ in the proof of Proposition~\ref{prop:no-4tree}.}
			\label{fig:no-4tree}
		\end{figure}
		It follows that $T^\prime$ contains the labelled graph $F \simeq K_{1,3} + 3K_1$ and therefore $G^\prime$ must contain the labelled graph $G \simeq K_{1,3} + K_1$, both described in \figurename~\ref{fig:no-4tree}, as an induced subgraph.
		
		Since $T^\prime \simeq \mathsf{TS}_2(G^\prime)$ is a tree and $G^\prime$ contains $G$, it follows that $G^\prime$ has exactly one non-trivial component $C$ (having more than two vertices) and $C$ contains $G$, otherwise $G^\prime$ must contain an induced $2K_2$ and by Proposition~\ref{keq2} its $\mathsf{TS}_2$-graph is not a tree, a contradiction.
		\begin{itemize}
			\item \textbf{Case~1: $a \in V(C)$.}
			By definition, the distance from $a$ to any of $b, c, d, e$ in $G^\prime$ must be at least two.
			If there is a path of length at least two between $a$ and one of $c, d, e$ not passing through $b$, the graph $G^\prime$ contains a $2K_2$, a contradiction.
			Thus, any path between $a$ and one of $c, d, e$ must go through $b$.
			Moreover, if there is a path of length at least three between $a$ and $b$ not passing through any of $c, d, e$, again the graph $G^\prime$ contains a $2K_2$, a contradiction.
			Since $a \in V(C)$, it follows that $a$ and $b$ must have a common neighbor in $G^\prime$, say $f$.
			Observe that for each $y \in V(C) - \{a,b,c,d,e,f\}$, $y$ must be adjacent to $b$ in $G^\prime$, otherwise $G^\prime$ either contains $2K_2$ or $D_{2,2,2}$ and again by Proposition~\ref{keq2} its $\mathsf{TS}_2$-graph is not a tree, a contradiction.
			However, this implies that $\mathsf{TS}_2(C)$ must be a forest and since $G^\prime$ has exactly one non-trivial component $C$, we have $\mathsf{TS}_2(G^\prime)$ is also a forest, a contradiction.
			
			\item \textbf{Case~2: $a \notin V(C)$.} 
			In this case, there are two types of size-$2$ stable sets of $G^\prime$: those containing $a$ and those do not.
			Since $G^\prime$ contains $G$, each type has at least one member.
			Moreover, since $a$ is isolated (the only non-trivial component is $C$ and $a$ is not in it), no member from one type is adjacent to a member from another type in $\mathsf{TS}_2(G^\prime)$, which means $\mathsf{TS}_2(G^\prime)$ is indeed disconnected, a contradiction.
		\end{itemize}
	
		In the above cases, we proved that some contradiction must happen.
		Our proof is complete.
	\end{itemize}
\end{proof}

Indeed, for $K_{1,n}$, in general we have

\begin{proposition}\label{prop:tree-contain-K1n}
	There exists a $\mathsf{TS}_k$-ary tree $T$ containing $K_{1,n}$ if $n \leq 2k$.
\end{proposition}

\begin{proof}
	From either \cite{AvisH22} or Proposition~\ref{stars}, the proposition holds for $n \leq k$.
	(Indeed, in this case, $T = K_{1,n}$.)
	Thus, it suffices to consider $k+1 \leq n \leq 2k$.
	For each $i \in \{1, \dots, n-k\}$, let $A_i = \{1, \dots, k\} - i$.
	
	Let $I_k = \{a_1, \dots, a_k\}$ and $B_n = \{b_1, \dots, b_n\}$.
	We construct a graph $G^0$ such that $\mathsf{TS}_k(G^0) \simeq K_{1,n} + (n-k)(k-1)K_1$.
	Let $I_k = \{a_1, \dots, a_k\}$ and $B_n = \{b_1, \dots, b_n\}$.
	Let $V(G) = I_k + B_n$.
	Vertices in $B_n$ form a graph $K_n - M$ where $M$ is the matching that contains $b_ib_{k+i}$ for $1 \leq i \leq n - k$.
	Additionally, for each $i \in \{1, \dots, k\}$, we add an edge in $G^0$ between $a_i$ and both $b_i$ and $b_{k+i}$.
	Observe that $V(\mathsf{TS}_k(G^0))$ consists of $I_k$, the sets $I_k - a_i + b_i$ ($1 \leq i \leq k$), $I_k - a_i + b_{k+i}$ ($1 \leq i \leq n - k$), and $(I_k - a_i + b_i) - a_j + b_{k+i}$ ($1 \leq i \leq n-k$ and $j \in A_i$).
	Moreover, one can verify that the independent sets $(I_k - a_i + b_i) - a_j + b_{k+i}$ are isolated in $\mathsf{TS}_k(G^0)$ and the remaining independent sets form a $K_{1,n}$ in which $I_k$ is adjacent to every other set.
	In short, $G^0$ is indeed our desired graph.
	
	For each $i \in \{1, \dots, n-k\}$, we construct a graph $G^i$ whose $\mathsf{TS}_k$-graph is a star $K_{1,k-1}$ as follows.
	Let $V(G^i) = (I_k - a_i + b_i) + \bigcup_{j \in \{1, \dots, k\} - i}\{c^i_j\}$.
	Vertices in $\bigcup_{j \in A_i}\{c^i_j\}$ form a clique in $G^i$ of size $k-1$.
	We also add an edge in $G^i$ between $a_j$ and $c^i_j$ for each $j \in A_i$. 
	From either~\cite{AvisH22} or Proposition~\ref{stars}, one can verify that $\mathsf{TS}_k(G^i) \simeq K_{1,k-1}$ as desired.
	For each $i \in \{1, \dots, n-k\}$ and $j \in A_i$, we construct a graph $G^i_j$ whose $\mathsf{TS}_k$-graph is a $K_2$ as follows.
	Let $V(G^i_j) = (I_k - a_i + b_i) - a_j + b_{k+i} + c^i_j$.
	The only edge in $G^i_j$ is the one joining $c^i_j$ and $b_{k+i}$.
	From either~\cite{AvisH22} or Proposition~\ref{paths}, one can verify that $\mathsf{TS}_k(G^i_j) \simeq K_2$ as desired. 
	
	\begin{figure}[!ht]
		\centering
		\includegraphics[width=0.7\textwidth]{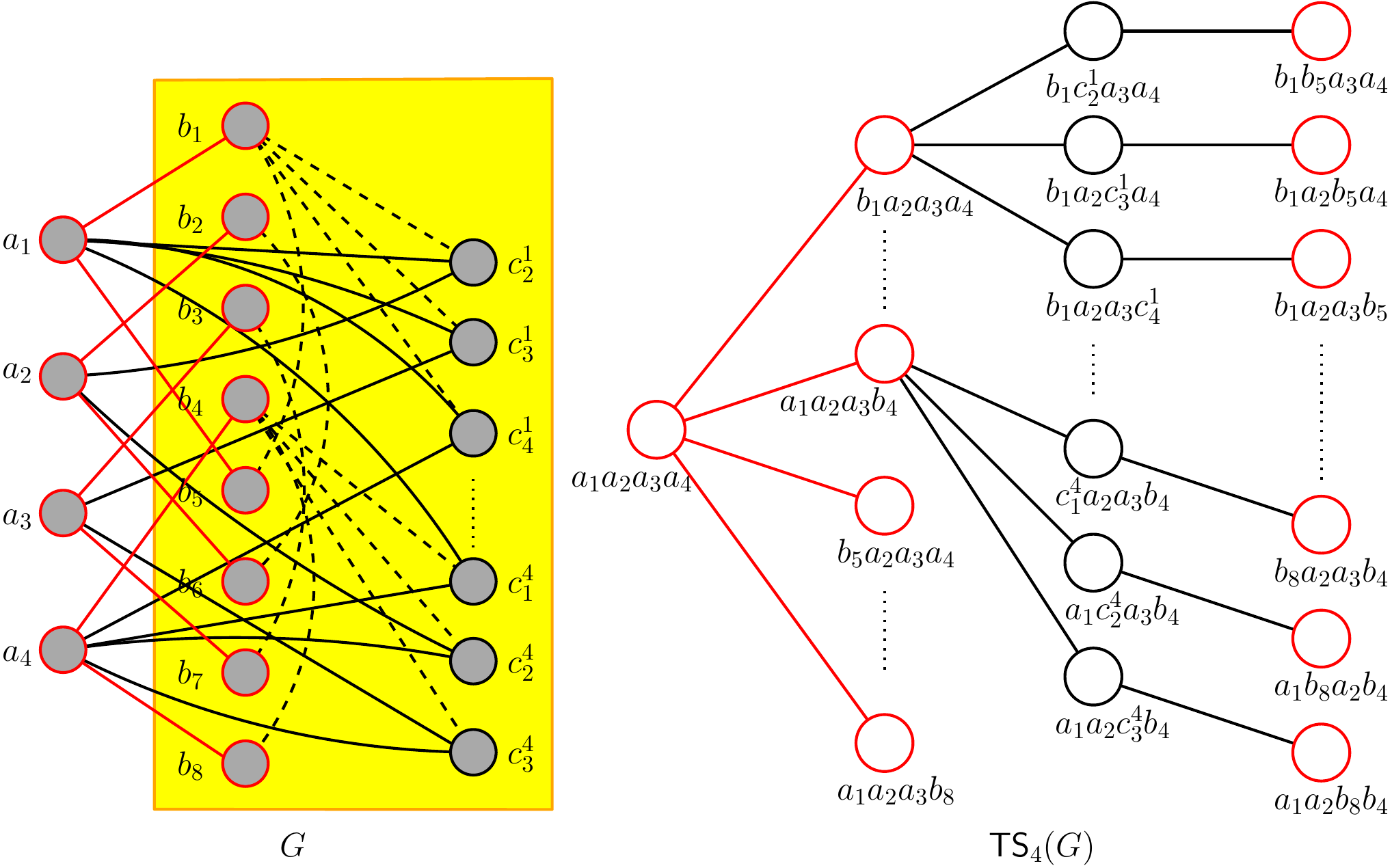}
		\caption{Construction of a graph $G$ such that $\mathsf{TS}_4(G)$ is a tree containing $K_{1,8}$. Vertices of $G$ in the yellow box form a clique having all dashed edges removed. The red induced subgraph of $G$ forms a graph $G^0$ whose $\mathsf{TS}_4(G^0) \simeq K_{1,8} + 12K_1$.}
		\label{fig:K18-in-TS4tree}
	\end{figure}
	
	Now, we construct a graph $G$ whose $\mathsf{TS}_k$-graph is a tree containing $K_{1,n}$ as follows. 
	For convenience, we assume that for each $i \in \{1, \dots, n-k\}$ the set $A_i = \{1, \dots, k\} - i$ can be enumerated as $\{j_1, \dots, j_{k-1}\}$.
	We define $\mathcal{K}^i_{j_0} = G^i$ and $\mathcal{K}^i_{j_p} = H_{j_p}(\mathcal{K}^i_{j_{p-1}}, G^i_{j_p})$ for $j_p \in A_i$ where $H_{j_p}$ is the stable set $(I_k - a_i + b_i) - a_{j_p} + c^i_{j_p}$ for $p \in \{1, \dots, k-1\}$.
	Observe that the graphs $\mathcal{K}^i_{j_{p-1}}$ and $G^i_{j_{p}}$ are $H_{j_p}$-consistent, which implies that $\mathcal{K}^i_{j_p}$ are well-defined.
	Moreover, one can also directly verify that the sets $(I_k - a_i + b_i) - a_j + c^i_j$ and $(I_k - a_{i^\prime} + b_{i^\prime}) - a_{j^\prime} + c^{i^\prime}_{j^\prime}$ always differ in at least two members, which means the condition (\ref{star}) of Proposition~\ref{overlap} is not satisfied.
	In short, for each $i \in \{1, \dots, n-k\}$, we obtain the graph $\mathcal{K}^i_{j_{k-1}}$ whose $\mathsf{TS}_k$-graph is isomorphic to the one obtained from $K_{1,k-1}$ by replacing each edge with a $P_3$.
	Next, we define $\mathcal{K}^0 = G^0$ and $\mathcal{K}^i = H^i(\mathcal{K}^{i-1}, G^{i})$ where $i \in \{1, \dots, n-k\}$ and $H^i$ is the subgraph induced by $(I_k - a_i + b_i) + b_{k+i}$.
	Observe that the graphs $\mathcal{K}^i$ are well-defined because $\mathcal{K}^{i-1}$ and $G^i$ are $H^i$-consistent.
	Moreover, we have $I_k$ and each $(I_k - a_i + b_i) - a_j + c^i_j$ for $1 \leq i \leq n-k$ and $j \in A_i$ always differ in at least two members.
	It follows that the condition (\ref{star}) of Proposition~\ref{overlap} is not satisfied.
	In short, we finally obtain the graph $G = \mathcal{K}^{n-k}$ whose $\mathsf{TS}_k$-graph is indeed a tree containing $K_{1,n}$ as desired.
\end{proof}

Unfortunately, we have not been able to show whether the reverse statement of Proposition~\ref{prop:tree-contain-K1n} also holds.
We conclude this section with the following open problems:

\begin{problem}
For every $k\ge 3$ and tree $T$, is there a graph $G$ such that 
$\Tk(G)$ is a forest containing $T$ as an induced subgraph?
\end{problem}
\begin{problem}
For every $k \ge 3$ and $(k+1)$-ary tree $T$, is there a graph $G$ such that $\Tk(G)$ is a tree
containing $T$ as an induced subgraph?
\end{problem}
\begin{problem}
	Does there exist a $\mathsf{TS}_k$-tree $T$ containing $K_{1,n}$ for $n > 2k$?
\end{problem}

\subsection{$D_{r,n,s}$}
We now consider graphs in the $D_{r,n,s}$ family for whose $\mathsf{TS}_k$-graphs
are trees and show how they can be
constructed by the $H$-join operation.
We remark that when $n = 1$, $D_{r,n,s}$ is nothing but a star $K_{1,r+s}$ and this case was considered in~\cite{AvisH22} and revisited in Proposition~\ref{stars}.
Furthermore, it follows from Proposition~\ref{k-ary tree} that
for $n,k \ge 2$ and $1 \leq r \leq s \le k-1$,
$D_{r,n,s}$ is a $k$-ary tree and so by Proposition~\ref{k-ary tree}
it is a $\Tk$-graph.
The reverse statement does not hold in general: there exists a $\mathsf{TS}_k$-graph $D_{r,n,s}$ even when $s \geq k$.
For example, one of such graphs, as already proved in~\cite{AvisH22}, is $D_{1,3,2}$ ($r = 1$, $s = k = 2$, and $n = 3$).
(See also \figurename~\ref{D132}.)
Indeed, as we will see in Proposition~\ref{prop:D1n2-TS2}, it is the unique $\mathsf{TS}_2$-graph among all trees $D_{1,n,2}$ for $n \geq 1$.
Additionally, for the sake of completeness, we will also show in Proposition~\ref{Dr2s} that the reverse
statement indeed holds when $n = 2$.

We are now characterizing which $D_{1,n,2}$-graphs are $\mathsf{TS}_2$-graphs and show
that this property is non-hereditary for
this simple class of trees. We then consider the $D_{r,2,s}$-graphs characterizing those that
are $\mathsf{TS}_k$-graphs.

Assume for some $G$, $\mathsf{TS}_2(G)$ is a forest containing a $K_{1,3}$. There are four stable sets
in $G$ corresponding to the vertices of the $K_{1,3}$. 
There are two ways of labelling the $K_{1,3}$ but in each case there are five vertices,
say $a, \dots, e$, of $G$ involved.
Up to permutations of the labels, the corresponding
stable sets in $G$ are either $\{ab,ac,bd,ae\}$ or $\{ab,ac,ad,ae\}$.
Using these definitions we have the following lemma.

\begin{lemma}
\label{H}
Let $H$ be the subgraph of $G$ induced by $a, b, \dots, e$.
The edges of $H$ are
\begin{itemize}
\item[(a)] $ad,de,eb,bc,cd$, if the $K_{1,3}$ is labelled $\{ab,ac,bd,ae\}$, or
\item[(b)] $bc,bd,be$ if the $K_{1,3}$ is labelled $\{ab,ac,ad,ae\}$.
\end{itemize}
\end{lemma}
\begin{figure}[htpb]
	\centering
	\includegraphics[width=0.4\textwidth]{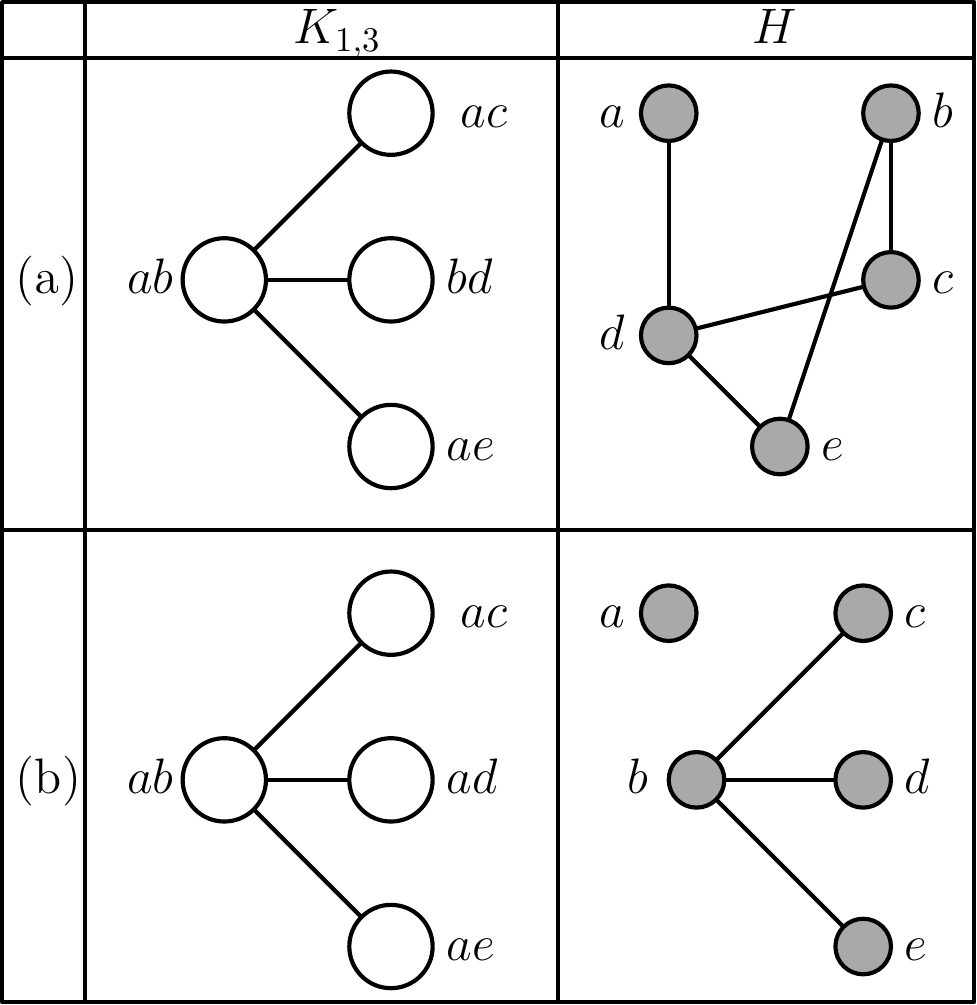}
	\caption{If $\mathsf{TS}_2(G)$ is a forest containing a $K_{1,3}$ then $G$ must contain
	one of the induced subgraphs $H$.}
	\label{fig:label-K13-TS2}
\end{figure}
\begin{proof}
\begin{itemize}
\item[(a)] 
This labelling of $K_{1,3}$ immediately gives edges $ad,bc,be$ and non-edges $ab,ac,ae,bd$. 
That leaves three edges of $H$ to be decided:
\begin{itemize}
\item[(i)] $ce$ must be a non-edge else there is an edge $ae,ac$ in the $K_{1,3}$.
\item[(ii)] $cd$ is an edge else there is a cycle $ab,bd,cd,ad$ in $\mathsf{TS}_2(G)$,
so it is not a tree.
\item[(iii)] $de$ is an edge else there is a cycle $de,bd,ab,ae$ in $\mathsf{TS}_2(G)$.
\end{itemize}
Note that $ce$ must also be a vertex in $\mathsf{TS}_2(G)$.
\item[(b)] 
This labelling of $K_{1,3}$ immediately gives edges $bc,bd,be$ and non-edges $ab,ac,ad,ae$.
There are no other edges in $H$ as $c,d,e$ form a stable set. This implies that
$\mathsf{TS}_2(G)$ must also contain vertices $cd,ce$ and $de$.
\end{itemize}
\end{proof}
Using the lemma we show that precisely one of the $D_{1,n,2}$-graphs is a $\mathsf{TS}_2$-graph,
incidentally proving the non-hereditary property mentioned above for this class of graphs.

\begin{proposition}\label{prop:D1n2-TS2}
	$D_{1,n,2}$ is a $\mathsf{TS}_2$-graph if and only if $n = 3$.
\end{proposition}
\begin{proof}
We first consider $1 \leq n \leq 3$ and show that $D_{1,3,2}$ is a $\mathsf{TS}_2$-graph while $D_{1,1,2} = K_{1,3}$ and $D_{1,2,2}$ are not.
(We note that the results for the first two graphs have also been proved in~\cite{AvisH22}.)
According to Lemma \ref{H}, if $D_{1,n,2}$ is a $\mathsf{TS}_2$-graph of some graph $G$, the unique star $K_{1,3}$ in $D_{1,n,2}$ can be labelled in one of two ways.
However, we may immediately eliminate the possibility of the labelling in Lemma \ref{H}(b).
This is because, as pointed out in the proof, there must be additional vertices in $D_{1,n,2} = \mathsf{TS}_2(G)$
labelled $cd,ce$ and $de$ which are non-adjacent since $c,d,e$ form a stable set in $G$.
This implies that $n \ge 6$. So we may assume that if $D_{1,n,2}$ is a $\mathsf{TS}_2$-graph, the $K_{1,3}$
must be labelled as in Lemma \ref{H}(a) with corresponding induced subgraph $H$ of $D_{1,n,2}$. 
From the proof of Lemma \ref{H}(a) there must be an additional vertex $ce$ in 
$D_{1,n,2}$ however this cannot be adjacent to any of the other four vertices.
This implies that $n \ge 3$ and so neither
$D_{1,1,2}$ nor $D_{1,2,2}$ can be  $\mathsf{TS}_2$-graphs. However we may extend $H$ to $G$ by adding a vertex
$f$ adjacent to all vertices except $e$, as illustrated in \figurename~\ref{D132}. 
This introduces the new stable set $ef$ which is adjacent
to both $ae$ and $ce$. Therefore $D_{1,3,2}$ is isomorphic to $\mathsf{TS}_2(G)$. 
We note that $G$ is the unique graph (up to label permutations) for which this is true, due to the
uniqueness of the labelling of $K_{1,3}$.

It remains to consider $n \geq 4$ and show that $D_{1,n,2}$ is not a $\mathsf{TS}_2$-graph.
Suppose to the contrary that there exists a graph $G$ such that $D_{1,n,2} = \mathsf{TS}_2(G)$.
Again, $D_{1,n,2}$ must contain a copy of $K_{1,3}$ with exactly two ways of labelling (up to label permutations) by size-$2$ independent sets of $G$.
\begin{itemize}
\item \textbf{Case 1: $K_{1,3}$ is labelled $\{ab, ac, bd, ae\}$.}
		Since $ac$ and $ae$ are not adjacent, $ce$ must be a vertex of $D_{1,n,2} = \mathsf{TS}_2(G)$.
		We consider the following cases:
		\begin{itemize}
			\item \textbf{Case 1.1: the distance between $ce$ and any vertex of $\{ac, bd, ae\}$ is at least three.}
			Since the roles of $c$ and $e$ are equal, we assume without loss of generality that $ce$ is adjacent to some vertex $cf$.
			Observe that $a$ and $f$ are not adjacent in $G$, otherwise $ac$ and $cf$ are adjacent, which means the distance between $ac$ and $ce$ is two, a contradiction.
			Since $ce$ and $cf$ are adjacent, so are $ae$ and $af$.
			Moreover, $bf$ must be a vertex, otherwise there is an edge between $ab$ and $af$ in $D_{1,n,2} = \mathsf{TS}_2(G)$ which creates a $C_3$ having $\{ab, ae, af\}$ as its vertex-set, a contradiction.
			Since $ab$ and $ac$ are adjacent, so are $cf$ and $bf$.
			Now, $df$ must be a vertex, otherwise $bd$ and $bf$ are adjacent which contradicts $D_{1,n,2} = \mathsf{TS}_2(G)$.
			Since $ab$ and $bd$ are adjacent, so are $af$ and $df$.
			From the proof of Lemma \ref{H}(a)(ii) $c$ and $d$ are adjacent in $G$, so $df$ and $cf$ are adjacent, which again contradicts $D_{1,n,2} = \mathsf{TS}_2(G)$.

		\item \textbf{Case 1.2: the distance between $ce$ and one of $\{ac, bd, ae\}$ is exactly two.} 
			Observe that $bd$ and $ce$ has no common neighbor, otherwise that neighbor must be labelled as one of $\{bc, be, dc, de\}$: the first two can be ignored because $ab$ and $ac$ (resp., $ab$ and $ae$) are adjacent, the last two can be ignored because $ab$ and $bd$ are adjacent.
			Again, since the roles of $c$ and $e$ are equal, we assume without loss of generality that $ae$ and $ce$ has a common neighbor $ef$.
			Since $n \geq 4$, $ce$ must have another neighbor which is different from $ef$, which can be either $cg$ or $eg$ for some vertex $g$ of $G$.
			\begin{itemize}
			\item If it is $cg$ then $ag$ must be a vertex, otherwise $cg$ and $ac$ must be adjacent, which creates a $C_6$ whose vertex-set is $\{ac, ab, ae, ef, ce, cg\}$, a contradiction.
				Since $ce$ and $cg$ are adjacent, so are $ae$ and $ag$, which contradicts $D_{1,n,2} = \mathsf{TS}_2(G)$.

			\item If it is $eg$ then $ag$ must be a vertex, otherwise $eg$ and $ae$ must be adjacent, which creates a $C_4$ whose vertex-set is $\{ae, ef, ce, eg\}$, a contradiction.
			Since $ce$ and $eg$ are adjacent, so are $ag$ and $ac$, which contradicts $D_{1,n,2} = \mathsf{TS}_2(G)$.

			\end{itemize}
		\end{itemize}
		
\item \textbf{Case 2: $K_{1,3}$ is labelled $\{ab, ac, ad, ae\}$.}
	As before, $cd$, $ce$, and $de$ must be vertices in $D_{1,n,2}$.
	Without loss of generality, since the roles of $c, d, e$ are equal, we may assume that only $ae$ is adjacent to another vertex of $D_{1,n,2}$. As shown in the proof of 
	Lemma \ref{H}(b), $D_{1,n,2}$ must also contain vertices $cd,ce,de$.
	Let $P$ be the path between $ae$ and $cd$.
	Since the roles of $c$ and $d$ are equal, we can assume without loss of generality that $cd$ is adjacent to a vertex $cf$ in $P$.
	Observe that if $af$ is not a vertex $ac$ and $cf$ are adjacent contradicting the choice of $ae$. So $af$ is a vertex and
	since $cd$ and $cf$ are adjacent so are $ad$ and $af$, which contradicts $D_{1,n,2} = \mathsf{TS}_2(G)$.
\end{itemize}
\end{proof}
We remark that if we add a vertex $g$ to $G$ in \figurename~\ref{D132} joining it to all
vertices except $d$ the corresponding $\mathsf{TS}_2$-graph is obtained by adding the edge
between $bd$ and $dg$ to $\mathsf{TS}_2(G)$. Note that this tree is not in the class $D_{r,n,s}$.

In the next proposition we consider two arbitrary stars whose centers are connected
by an edge.
\begin{proposition}
\label{Dr2s}
	$D_{r,2,s}$ ($1 \leq r \leq s$) is a $\mathsf{TS}_k$-graph if and only if $s \leq k-1$.
\end{proposition}
\begin{proof}
	\begin{itemize}
		\item[($\Leftarrow$)] It follows directly from Proposition~\ref{k-ary tree}.
	
		\item[($\Rightarrow$)] Suppose that $D_{r,2,s}$ ($r \leq s$) is obtained from $P_2 = p_1p_2$ by attaching $r$ leaves $u_1, \dots, u_r$ at $p_1$ and $s$ leaves $v_1, \dots, v_s$ at $p_2$ for some $s \geq k$.
		We show that this graph is not a $\mathsf{TS}_k$-graph for any fixed $k \geq 2$.
		Suppose to the contrary that there exists a graph $G$ such that $D_{r,2,s} \simeq \mathsf{TS}_k(G)$, i.e., there exists a bijective mapping $f: V(D_{r,2,s}) \to V(\mathsf{TS}_k(G))$ such that $uv \in E(D_{r,2,s})$ if and only if $f(u)f(v) \in E(\mathsf{TS}_k(G))$.
		Without loss of generality, let $f(p_2) = I = \{a_1, \dots, a_k\}$, where $I$ is a size-$k$ independent set of $G$.
		Since $p_2$ has $s + 1$ neighbors, from the pigeonhole principle, it follows that there must be some $i \in \{1, \dots, k\}$ such that $f(u) = I - a_i + x$ and $f(v) = I - a_i + y$, where $u, v \in N(p_2)$.
		Observe that $J = (I - a_i - a_j) + x + y \notin \{f(p_2), f(u), f(v)\}$ must be a size-$k$ independent set of $G$, where $j \in \{1, \dots, k\} - i$ and therefore there exists $z \in V(D_{r,2,s}) - \{p_2, u, v\}$ such that $f(z) = J$.
		We consider the following cases:
		\begin{itemize}
			\item \textbf{Neither $u$ nor $v$ is $p_1$.}
			In this case, we must have $z \notin N(p_2)$, otherwise it must be adjacent to $p_2$, but then $f(z) = J$ and $f(p_2) = I$ must be adjacent in $\mathsf{TS}_k(G)$, a contradiction.
			It follows that $z \in N(p_1) - p_2$ and thus $f(p_1)$ must be in $\{I - a_i + x, I - a_i + y, I - a_j + x, I - a_j + y\}$.
			Since neither $u$ nor $v$ is $p_1$, the first two can be ignored.
			Now, if $f(p_1) = I - a_j + x$, the vertices $x$ and $a_j$ must be adjacent in $G$, which contradicts the fact that $f(u) \in \mathsf{TS}_k(G)$.
			A similar contradiction can be derived for the case $f(p_1) = I - a_j + y$.
			Thus, $f(p_1)$ cannot be defined.
			\item \textbf{$u$ is $p_1$.} Again, $z \notin N(p_2)$. 
			Thus, $z \in N(p_1) - p_2$, which implies that $y$ and $a_j$ must be adjacent in $G$.
			This contradicts $f(v) \in \mathsf{TS}_k(G)$.
			Thus, $f(z)$ cannot be defined.
		\end{itemize}
	In both cases, we showed that some contradiction must occur.
	Our proof is complete.
	\end{itemize}
\end{proof}

\section{Conclusions}
In this paper, we 
considered two token sliding problems for trees and forests.
The two questions studied seem remarkably complicated, even for 
this simple class of graphs.
For the first question, finding necessary and sufficient conditions on $G$ for $\Tk(G)$ to be a 
forest, we could only get a complete solution for $k=2,3$.
For the second question, finding necessary and sufficient conditions for a tree or forest to
be a token sliding graph, we could get more general results. Nevertheless, as noted in
Section \ref{Q2}
several interesting important questions remain.
We expect the join and decomposition operations introduced there will be of use for
similar questions for more general graphs.

\section*{Acknowledgments}

Avis' research is partially supported by the Japan Society for the Promotion of Science (JSPS) KAKENHI Grants JP18H05291, JP20H00579, and JP20H05965 (AFSA) and Hoang's research by JP20H05964 (AFSA).

\bibliography{refs.bib}
\end{document}